\documentclass[10pt]{amsart}
\usepackage{amsmath,amsthm,amsfonts,amssymb,amscd}
\usepackage[all]{xy}

\usepackage[english]{babel}
\usepackage{graphicx,color}
\usepackage[T1]{fontenc}
\usepackage[latin1]{inputenc}
\usepackage{ae}
\usepackage{enumerate}

\usepackage[mathscr]{eucal}
\usepackage{amsmath,amssymb}
\usepackage{graphics}
\usepackage{multirow}
\usepackage{tikz}
\usetikzlibrary{arrows.meta,decorations.pathmorphing,calc}
\usepackage{hyperref}
\IfFileExists{srcltx.sty}{\usepackage[active]{srcltx}}{}
\usepackage[a4paper,top=3cm,bottom=3cm,left=3cm,right=3cm]{geometry}

\newtheorem{thm}{Theorem}[section]
\newtheorem{THM}{Theorem}

\newtheorem{prop}[thm]{Proposition}

\newtheorem{lemma}[thm]{Lemma}

\theoremstyle{definition}

\newtheorem{definition}[thm]{Definition}
\newtheorem{remark}[thm]{Remark}

\newtheorem{question}[thm]{Question}

\newcommand{\nc}{\newcommand}
\nc {\hh}{\check{h}}
\nc {\DD}{\mathcal{D}}
\nc {\RR}{\mathcal{R}}
\nc {\Pp}{\mathbb{P}}
\nc {\Ss}{\mathcal{S}}
\nc {\PP}{\mathbb{P}^{2}}
\nc {\Pd}{ \check{\mathbb{P}}^{2}}
\nc {\WW}{\mathcal{W}}
\nc {\Sym}{\mathrm{Sym}}
\nc {\OO}{\mathcal{O}}
\nc {\CC}{\mathbb{C}}
\nc {\EE}{\mathcal{E}}
\nc {\MM}{\mathcal{M}}
\nc {\KK}{\mathcal{K}}
\nc {\PW}{\mathcal{P}}
\nc {\NW}{\mathcal{N}_{\WW}}
\nc {\FF}{\mathcal{F}}
\nc {\GG}{\mathcal{G}}
\nc {\ZZ}{\mathcal{Z}}
\nc {\LL}{\mathcal{L}}
\nc {\HH}{\mathcal{H}}
\nc {\NN}{\mathcal{N}}
\nc {\VV}{\mathcal{V}}
\nc {\Ww}{\mathbb{W}}
\nc {\QQ}{\mathbb{Q}}
\nc {\II}{\mathcal{I}}
\nc {\tang}{\mathrm{Tang}}
\nc {\Diff}{\mathrm{Diff}}
\nc {\Diffun}{\mathrm{Diff}^1(\CC^2,0)}
\nc {\pr}{\mathrm{pr}}
\nc {\cod}{\mathrm{codim}}
\nc {\id}{\mathrm{id}}
\nc {\ord}{\mathrm{ord}}
\nc {\Fol}{\mathrm{Fol}}
\nc {\Distr}{\mathrm{Distr}}
\nc {\rank}{\mathrm{rank}}
\nc {\Lie}{\mathrm{Lie}}
\nc {\Z}{\mathbb{Z}}
\nc {\trdeg}{\mathrm{tr.deg}}

\begin{document}

\title{Neighborhoods of curves with a prescribed number of foliations}

\author[M. Falla Luza, R. Rosas]
{Maycol Falla Luza$^{1}$, Rudy Rosas$^2$}
\address{\newline $1$ Universidade Federal Fluminense, Rua M\'ario Santos Braga S/N, Niter\'oi, RJ, Brazil.\hfill\break
$2$ Pontificia Universidad Cat\'olica del Per\'u, Av. Universitaria 1801, San Miguel
15088, Per\'u}
\email{$^1$ hfalla@id.uff.br} \email{$^2$ rudy.rosas@pucp.pe}
\date{\today}


\begin{abstract}
Given a connected projective curve $C \subset \mathbb{P}^n$, $n\geq 2$, and an integer $0\leq \ell \leq n$, we construct an $n$-dimensional (non compact) complex manifold, obtained as a neighborhood of an embedded copy of $C$, which carries exactly $\ell$ codimension one holomorphic foliations; moreover, every codimension one distribution on it is one of these foliations. We also determine the field of meromorphic functions of these manifolds: it can be prescribed to be $\CC$ or a purely transcendental extension of transcendence degree one, and no larger field is possible as soon as the number of foliations is finite. This extends to arbitrary dimension, and refines, previous constructions of neighborhoods of curves in surfaces without foliations or without non-constant meromorphic functions.
\end{abstract}

\maketitle

\section*{Introduction}

Let $C$ be a compact holomorphic curve embedded in a complex manifold $S$. The germ of $S$ along $C$ -- that is, the class of arbitrarily small neighborhoods of $C$ in $S$ -- is a classical object of study, going back to the work of Grauert \cite{Gra}. Recently there has been much interest in understanding how ``generic'' such germs can be, in the following sense: which analytic objects (functions, foliations, curves) do arbitrary neighborhoods of $C$ carry?  In \cite{FL}, the first author and F.~Loray constructed, for each $d>0$, germs of smooth complex surfaces $(S,C)$ containing a smooth rational curve $C$ of self-intersection $d$ and admitting no singular holomorphic (or even formal) foliation; since every non-constant meromorphic function $f$ on a neighborhood of $C$ induces the singular holomorphic foliation $\{df=0\}$, such germs carry no non-constant meromorphic function. In \cite{FLS}, examples of non-algebraizable neighborhoods of curves in surfaces were produced by gluing techniques, and the field of meromorphic functions of the resulting surfaces was computed: its transcendence degree over $\CC$ can be prescribed to be $0$, $1$ or $2$. See also the papers of Lvovski \cite{Lv1,Lv2} for closely related constructions and for the study of fields of meromorphic functions on neighborhoods of rational curves, and \cite{FLP} for higher dimensional examples of submanifolds with ample normal bundle whose germs of neighborhoods have no non-constant meromorphic functions.

In the present paper we consider neighborhoods of curves in manifolds of arbitrary dimension $n\geq 2$ and we take as our main invariant the \emph{number of codimension one holomorphic foliations} defined on the whole neighborhood. All foliations and distributions in this paper are singular: a codimension one distribution $\DD$ on a complex manifold $S$ is locally defined by holomorphic $1$-forms $\omega$ with singular set of codimension at least two, up to multiplication by units, and $\DD$ is a foliation when $\omega \wedge d\omega=0$. We denote by $\Distr(S)$ and $\Fol(S)$ the sets of codimension one distributions and foliations on $S$.

Our curves are allowed to be singular, and we use the following terminology. Given a projective curve $C\subset\Pp^n$ and a complex manifold $S$ of dimension $n$, a compact analytic curve $\widehat C\subset S$ is an \emph{embedded copy} of $C$ if there is a biholomorphism $\iota:C\to\widehat C$ which is locally induced by ambient biholomorphisms: every point of $C$ has an open neighborhood $\Omega\subset\Pp^n$ and there is a biholomorphism from $\Omega$ onto an open subset $\Omega'\subset S$ carrying $C\cap\Omega$ onto $\widehat C\cap\Omega'$ and restricting to $\iota$ there. In particular $C$ and $\widehat C$ are isomorphic as analytic spaces -- with the same singularities, if any -- and their germs of neighborhoods are locally isomorphic; when $C$ is smooth, this is the usual notion of an embedding of $C$ into $S$.

We shall repeatedly use the following form of the identity theorem, which we state once and for all.

\smallskip
\noindent$(\star)$ \emph{If two codimension one distributions on a connected complex manifold $S$ coincide on some non empty open subset, then they coincide on $S$; and if such a distribution is integrable on some non empty open subset, then it is integrable on $S$.}
\smallskip

Our main result shows that any number of foliations between $0$ and $n$ can be realized on a neighborhood of any projective curve.

\begin{THM}\label{thm:A}
Let $C \subset \mathbb{P}^n$, $n\geq 2$, be a connected projective curve, possibly singular, and let $\ell \in \{0, 1, \ldots, n\}$. Then there exist a smooth complex manifold $S$ of dimension $n$ and an embedded copy $\widehat{C}\subset S$ of $C$ such that:
\begin{enumerate}[(1)]
\item $S$ carries exactly $\ell$ codimension one holomorphic foliations. Moreover, every codimension one distribution on $S$ is integrable, so that
$$\#\Distr(S) = \#\Fol(S) = \ell.$$
\item The same holds for the germ of $S$ along $\widehat{C}$: for every connected open neighborhood $S'\subset S$ of $\widehat{C}$, the restriction maps $\Fol(S) \to \Fol(S')$ and $\Distr(S) \to \Distr(S')$ are bijections.
\item If $C$ is smooth, then $N_{\widehat{C}|S} \simeq N_{C|\mathbb{P}^n}\otimes \OO_C(p)$ for some point $p \in C$; in particular
$$\deg N_{\widehat{C}| S}=\deg N_{C| \mathbb{P}^n}+n-1=(n+1)d+2g+n-3,$$
where $d$ and $g$ denote the degree and the genus of $C$. For $n=2$ this reads $\widehat C\cdot \widehat C= d^2+1$.
\end{enumerate}
\end{THM}

The construction is, roughly speaking, as follows. We take a neighborhood of the curve $C \subset \mathbb{P}^{n}$ and a neighborhood of the exceptional divisor $E$ of the blow-up of the origin of $\mathbb{C}^{n}$. We then choose points $p \in C$ and $q\in E$ and glue these neighborhoods by means of a biholomorphism $\tilde{h}$ that sends a neighborhood of $q$ onto a region attached to $C$ at $p$. The gluing is performed so that the curve and the exceptional divisor become transversal in the resulting manifold; finally, we contract the exceptional divisor, thus obtaining a new neighborhood of the curve $C$. The gluing map is built from two one-dimensional ingredients with wild boundary behavior: a covering map of the punctured disc, whose iterates are linearly independent, and a uniformization of the interior of a Jordan curve without $C^1$-arcs. The first ingredient produces coordinates in which prescribed foliations become coordinate foliations while all remaining coordinate foliations become inextensible; the second one destroys, in a very strong sense, every distribution which is not a coordinate distribution. Depending on the choice of the local data one obtains various situations -- for instance, manifolds carrying only finitely many foliations or distributions.

Our second result concerns the field $\MM(S)$ of meromorphic functions of the manifolds of Theorem \ref{thm:A}. As recalled above, non-constant meromorphic functions induce foliations; more precisely, if $f,g \in \MM(S)$ satisfy $df\wedge dg \not\equiv 0$, then the pencil $\{f + t g\}_{t\in \CC}$ induces infinitely many pairwise distinct foliations on $S$ (Proposition \ref{prop:pencil}). Consequently, on a manifold with finitely many foliations, any two meromorphic functions must be functionally dependent. For the manifolds of Theorem \ref{thm:A} we obtain a complete description.

\begin{THM}\label{thm:B}
Let $S$ be a manifold as in Theorem \ref{thm:A}, together with its distinguished set of foliations $\Fol(S)=\{\FF_1, \ldots, \FF_{\ell}\}$. Then:
\begin{enumerate}[(1)]
\item If $\ell = 0$, then $\MM(S) = \CC$.
\item If $\ell \geq 1$, then, according to choices made in the construction, either
$$\MM(S)=\CC \qquad \text{or}\qquad \MM(S)=\CC(f),$$
a purely transcendental extension of $\CC$ of transcendence degree one, where $f$ is a meromorphic first integral of the foliation $\FF_1$; both cases occur for every $\ell\geq 1$. In the second case, none of the foliations $\FF_2, \ldots, \FF_{\ell}$ admits a non-constant meromorphic first integral on $S$.
\item In all cases $\trdeg_{\CC}\MM(S)\leq 1$, and statements (1) and (2) also hold for every connected open neighborhood of $\widehat{C}$ in $S$.
\end{enumerate}
\end{THM}

In particular we obtain, in any dimension $n \geq 2$ and for any projective curve $C$, neighborhoods of $C$ all of whose meromorphic functions are constant, in the spirit of \cite{FL, FLP}. Note also that the value $2$ for the transcendence degree, which is realized in the examples of \cite{FLS}, is impossible here: by Proposition \ref{prop:pencil}, a neighborhood of a curve whose field of meromorphic functions has transcendence degree at least two carries infinitely many foliations. In this sense Theorem \ref{thm:B} is optimal.

It is worth comparing Theorem \ref{thm:A} with the global, projective situation, which turns out to be completely different: a projective manifold of dimension at least two always carries \emph{infinitely many} codimension one foliations, since its field of rational functions has transcendence degree at least two and pencils of rational functions produce infinitely many foliations (Proposition \ref{prop:pencil}). Finiteness of the set of foliations is therefore a genuinely non-algebraic phenomenon, and germs of neighborhoods are its natural habitat. In the projective category the significant question is instead how many codimension one foliations can contain a \emph{fixed} foliation $\GG$ of codimension $q$: by the work of G.~S.~Barbosa and J.~V.~Pereira on unlikely intersections \cite[Theorem D]{BP}, either there are at most $q$ of them, or there are infinitely many. We discuss this in Remark \ref{rem:BP}, and Section \ref{sec:pereira} exhibits a codimension two foliation on $\Pp^3$ contained in exactly two codimension one foliations, showing that the bound is attained. Finally, Proposition \ref{prop:genposbound} shows that the bound $\#\Distr(S)\leq \dim S$ holds for distributions on an arbitrary complex manifold, provided $\Distr(S)$ is finite.

The paper is organized as follows. In Section \ref{sec:local}, we establish the local constructions that will be used to define the gluing map $\tilde{h}$ with the desired properties. In Section \ref{sec:gluing} we perform the gluing, prove that every distribution on the resulting manifold is a coordinate distribution in a suitable chart (Theorem \ref{thm:coordinate}), and discuss the general questions on manifolds with finitely many foliations, including the projective case (Remark \ref{rem:BP}) and the general position property (Proposition \ref{prop:genposbound}). Section \ref{sec:proofA} contains the proof of Theorem \ref{thm:A} and Section \ref{sec:functions} the proof of Theorem \ref{thm:B}. We close the paper, in Section \ref{sec:pereira}, with a different construction, explained to us by Jorge Vit\'orio Pereira, of a surface containing a rational curve and carrying exactly two foliations.

Throughout the paper we write $n = k+1$, so that a neighborhood of a curve in $\mathbb{P}^{k+1}$ has $k$ ``normal'' directions.

\section{Local constructions}\label{sec:local}

We denote by $\mathbb{D}_r\subset \CC$ the disc of radius $r$ centered at the origin, $\mathbb{D}:=\mathbb{D}_1$, $\mathbb{D}^*:=\mathbb{D}\setminus \{0\}$, and by $\mathbb{H}$ the upper half plane.

Let us start with the function $\sigma: \mathbb{D} \to \mathbb{D}^*$ given by the composition of the biholomorphism $\tau: \mathbb{D} \to \mathbb{H}$, $\tau(w)= i \dfrac{1+w}{1-w}$, and the covering $z\mapsto \exp(iz): \mathbb{H}\to \mathbb{D}^*$. This map is a covering which extends continuously to $\partial \mathbb{D}-\{1\}$, with $\sigma(\partial \mathbb{D}-\{1\})= \partial \mathbb{D}$, and it has the following property, which is the source of all the wild behavior exploited in this paper: \emph{given any $L \in \mathbb{D}$, there exists a sequence $(z_n)\to 1$ in $\mathbb{D}$ such that $\sigma(z_n)\to L$}. Indeed, $\tau$ maps any neighborhood of $1$ in $\mathbb{D}$ onto a neighborhood of infinity in $\mathbb{H}$, on which $\exp(iz)$ assumes every value of $\mathbb{D}^*$. We denote by $\sigma^j$ the composition $\sigma^{\circ j}$, and set $\sigma^0=\id$.

\begin{lemma}\label{lem:indep}
Given $s \in \mathbb{N}$, the functions $1, \id, \sigma, \sigma^2, \ldots, \sigma^s$ are linearly independent over $\CC$ as functions on $\mathbb{D}$.
\end{lemma}

\begin{proof}
By contradiction, take $s\geq 1$ minimal such that there is a non trivial linear relation
$$
a_{-1}+a_0x + a_1 \sigma(x) +\ldots +a_s \sigma^s(x) = 0, \qquad x \in \mathbb{D}.
$$
By minimality we necessarily have $a_s \neq 0$. Take $L\in \mathbb{D}$ arbitrary and $(z_n)\to 1$ such that $\sigma(z_n)\to L$; note that $\sigma^{j}(z_n)=\sigma^{j-1}(\sigma(z_n))\to \sigma^{j-1}(L)$ for $j\geq 1$, by continuity of $\sigma^{j-1}$ at $L\in \mathbb{D}$. Applying the relation to $z_n$ and taking the limit we obtain
$$(a_{-1}+a_0)+a_1L + a_2 \sigma(L) +\ldots + a_s \sigma^{s-1}(L)=0$$
for every $L \in \mathbb{D}$. Since $a_s\neq 0$, this is a non trivial relation of shorter length, which contradicts the minimality of $s$ if $s\geq 2$, and gives $a_1=0$ directly if $s=1$.
\end{proof}

Set now $f_j: \mathbb{D} \to \mathbb{D}_{\frac{1}{2}}$, $f_j(x) = \dfrac{1}{4}\left(\int_0^x \sigma^j(s)ds +1\right)$, where the integral is taken along any path in $\mathbb{D}$; since $\vert\int_0^x \sigma^j\vert<1$ we have $f_j(x) \neq 0$ for all $x\in \mathbb{D}$. Now, for any $s \in \mathbb{N}$, we define
\begin{equation}\label{phi}
\Phi=\Phi_s: \mathbb{D}\times \mathbb{D}^s \to  \mathbb{D}_{\frac{1}{2}}\times\mathbb{D}_{\frac{1}{2}}^{s}, \,\, \Phi(x, y_1, \ldots, y_s)=\left(\dfrac{x}{2}, y_1f_1(x), \ldots, y_s f_s(x)\right).
\end{equation}
\begin{prop}\label{Enunciado sobre Phi}
The map  $\Phi$ has the following properties.
\begin{enumerate}
\item $\Phi$ is a biholomorphism onto its image $U = \Phi( \mathbb{D}\times \mathbb{D}^s) \subset \mathbb{D}_{\frac{1}{2}}\times \mathbb{D}_{\frac{1}{2}}^s$, and it extends continuously to $\overline{\mathbb{D}}\times \overline{\mathbb{D}}^s$. Moreover $\Phi(\{y=0\})=\{y=0\}\cap U$.
\item If $\mathcal{F}$ is a codimension one foliation on $U$ tangent to the vector field $\dfrac{\partial \Phi}{\partial x}$, then there is a sequence $(Y_n)\subset \mathbb{D}^s$, $Y_n \to 0$, such that $\mathcal{F}$ does not extend to any neighborhood of $\Phi(1, Y_n)$. In particular $\mathcal{F}$ does not extend to any neighborhood of $\overline{U}$.
\end{enumerate}
\end{prop}

\begin{proof}
The fact that $f_j(x)\neq 0$ for every $x\in \mathbb{D}$ implies that $\Phi$ is injective, hence invertible over its image; the statements about $\{y=0\}$ are immediate from \eqref{phi} and from $f_j\neq 0$. Note that $|f_j'(x)|=\dfrac{1}{4}|\sigma^j(x)|<1$, thus $f_j$ is Lipschitz and extends continuously to $\overline{\mathbb{D}}$; this ends the first item.

For the second part, assume by contradiction that $\mathcal{F}$ extends to a neighborhood of every point $\Phi(1,Y)$ with $Y$ close to $0$; we will get a contradiction using only extendability near the points $\Phi(1,Y)$ with $Y$ in a small polydisc around the origin. Denote $a_j= f_j(1)$ (the continuous extension) and let us prove that $a_j\neq 0$. Since $\sigma^j$ maps $[0,1)$ into $\mathbb{D}$ and extends continuously to $[0,1]$, we have $\delta_j:=\int_0^1 \vert \sigma^j(t) \vert dt <1$; computing the limit of $f_j$ at $1$ along the radius $[0,1)$ we obtain
$$ \vert a_j\vert =\lim_{t \to 1^-} \vert f_j(t) \vert =\frac{1}{4} \lim_{ t \to 1^-} \left\vert \int_0^t \sigma^j(s) ds +1 \right\vert \geq \frac{1}{4} (1 - \delta_j) >0 .$$
Thus $\Phi(1, y)=(\frac{1}{2}, a_1y_1, \ldots, a_sy_s)$ maps $\{1\}\times \mathbb{D}^s$ onto $\{\frac{1}{2}\}\times \Delta$ for some polydisc $\Delta$ containing $0$. We take a point $q=(a_1y_1^0, \ldots, a_sy_s^0) \in \Delta$ with all $y_j^0\neq 0$ and such that $p=(\frac{1}{2},q)$ is a regular point of (the extension of) $\mathcal{F}$, and observe that, by hypothesis, the curve
$$
C(t) =\left(\frac{t}{2}, y_1^0f_1(t), \ldots, y_s^0f_s(t)\right), \,\, t\in \mathbb{D}
$$
is tangent to $\mathcal{F}$. Since $p$ is regular and $C(t) \to p$ as $t \to 1$, there is a connected neighborhood $V$ of $1$ in $\overline{\mathbb{D}}$ such that $C(t)$ lies in the plaque of $\mathcal{F}$ through $p$ for all $t\in V\cap \mathbb{D}$: indeed, for $t$ close to $1$ the connected curve $C(t)$ is tangent to $\mathcal{F}$ inside a foliation box around $p$, hence contained in a single plaque, whose closure contains $p$. Take $L\in \mathbb{D}$ and a sequence $(t_n)\to 1$ such that $\sigma(t_n)\to L$; then
$$
\lim_{n\to \infty} C'(t_n) = \left(\frac{1}{2}, \frac{y_1^0}{4}L, \frac{y_2^0}{4}\sigma(L), \ldots, \frac{y_s^0}{4}\sigma^{s-1}(L)\right) \in T_p \mathcal{F},
$$
because $C'(t)=\left(\frac12, \frac{y^0_1}{4}\sigma(t),\ldots, \frac{y^0_s}{4}\sigma^s(t)\right)$ is tangent to the leaves and tangent spaces vary continuously. Since $T_p\mathcal{F}$ is a hyperplane, the family of vectors above, parametrized by $L \in \mathbb{D}$, satisfies a non trivial linear relation
$$
\frac{\alpha}{2} + \beta_1\frac{y^0_1}{4} L + \beta_2\frac{y^0_2}{4}\sigma(L)+\ldots+\beta_s \frac{y^0_s}{4}\sigma^{s-1}(L)=0 \qquad \text{for all } L \in \mathbb{D},
$$
and since all $y^0_j\neq 0$ this contradicts Lemma \ref{lem:indep}.
\end{proof}

The following proposition produces parameterizations with inextensible coordinate foliations; it will give the case $\ell=0$ of Theorem \ref{thm:A}. We denote by $A:=\{(z_1,0,\dots,0)\}$ the first coordinate axis of $\CC^{k+1}$, and we use coordinates $(z_1,\dots,z_{k+1})=(x,y_1,\dots,y_k)$.

\begin{prop}\label{ell=0}
Let $\Delta_2 \subset \mathbb{C}^{k+1}$ be an open set, $p\in \Delta_2$ and let $C\subset \Delta_2$ be a smooth curve through $p$. Then there is a biholomorphism $\Psi: \mathbb{D}^{k+1}\to \Delta$, where $p \in \Delta \subset \overline{\Delta}\subset \Delta_2$ and $\Psi(0)=p$, such that:
\begin{enumerate}[(i)]
\item $\Psi^{-1}(C\cap \Delta)= A \cap \mathbb{D}^{k+1}$;
\item none of the $k+1$ foliations $\Psi_*\{dz_i=0\}$, $i=1,\dots,k+1$, defined on $\Delta$, extends to $\Delta_2$.
\end{enumerate}
\end{prop}

\begin{proof}
Consider first a biholomorphism $Z:\mathbb{D}_2^{k+1} \to \Delta'$ onto a neighborhood $\Delta'$ of $p$, with $\overline{\Delta}'\subset \Delta_2$, $Z(0)=p$ and $Z^{-1}(C\cap \Delta')=A\cap \mathbb{D}_2^{k+1}$; such a $Z$ exists because $C$ is smooth at $p$. Consider also the transposition $T(x,y_1, \ldots, y_k)=(y_1, x, y_2, \ldots, y_k)$ and define
$$\Psi= Z\circ T \circ \Phi \circ T \circ \Phi: \mathbb{D}^{k+1}\to \Delta_2,$$
with $\Phi=\Phi_k$ as in \eqref{phi}. Observe that $\Delta:=\Psi(\mathbb{D}^{k+1})$ satisfies $\overline{\Delta} \subset \Delta_2$, since the image of $T\circ \Phi\circ T\circ \Phi$ has compact closure in $\mathbb{D}_2^{k+1}$.

Let us check (i). By Proposition \ref{Enunciado sobre Phi}(1) we have $\Phi^{-1}(A)=A$ (as germs at the origin of $\mathbb{D}^{k+1}$, and in fact globally on $\mathbb{D}^{k+1}$); moreover $\Phi^{-1}(A_2)=A_2$, where $A_2$ denotes the second axis, because $\Phi(x,y)\in A_2$ forces $x=0$ and $y_j f_j(0)=0$ for $j\geq 2$. Since $T$ interchanges $A$ and $A_2$, we conclude that $(T\circ\Phi\circ T\circ \Phi)^{-1}(A)=A$, and (i) follows from $Z^{-1}(C\cap\Delta')=A\cap \mathbb{D}_2^{k+1}$.

For (ii), since $Z\circ T: \mathbb{D}_2^{k+1} \to \Delta_2$ is a biholomorphism onto its image, it is enough to see that $(\Phi \circ T \circ \Phi)_* \{dx=0\}$ and $(\Phi \circ T \circ \Phi)_* \{dy_j=0\}$ do not extend to $\mathbb{D}_2^{k+1}$ for all $j$. Let us explain why $\Phi$ has to be composed twice. By Proposition \ref{Enunciado sobre Phi}(2), a single $\Phi$ destroys the $k$ foliations $\{dy_j=0\}$, whose images are tangent to $\partial\Phi/\partial x$; but it does not destroy $\{dx=0\}$, whose image is again $\{dx=0\}$. The transposition $T$ turns $\{dx=0\}$ into a foliation of the first kind, so that the outer $\Phi$ destroys it as well, while for the $\{dy_j=0\}$ the damage already done by the inner $\Phi$ is simply carried along by the outer biholomorphism. Throughout, we write $U:=\Phi(\mathbb{D}^{k+1})$, a connected open set with $\overline U\subset \overline{\mathbb{D}}_{\frac12}^{k+1}\subset \mathbb{D}^{k+1}$.
\begin{enumerate}[(a)]
\item Observe that $\Phi_* \{dx=0\}=\{dx=0\}$ and hence $T_*\Phi_*\{dx=0\}=\{dy_1=0\}$, so that $(\Phi \circ T \circ \Phi)_* \{dx=0\}$ is the restriction of $\Phi_* \{dy_1=0\}$ to the open subset $\Phi(T(U))\subset U$. Since $\{dy_1=0\}$ is tangent to $\partial_x$, the foliation $\Phi_*\{dy_1=0\}$ is tangent to $\dfrac{\partial \Phi}{\partial x}$, so by Proposition \ref{Enunciado sobre Phi}(2) it extends to no neighborhood of $\overline{U}$. Now if $(\Phi \circ T \circ \Phi)_* \{dx=0\}$ extended to a distribution $\GG$ on $\mathbb{D}_2^{k+1}$, then $\GG|_U$ and $\Phi_*\{dy_1=0\}$ would be two distributions on the connected open set $U$ agreeing on $\Phi(T(U))$, hence equal by $(\star)$; as $\overline U$ is a compact subset of $\mathbb{D}_2^{k+1}$, this would give the forbidden extension.

\item Set $V:=\Phi_*\{dy_j=0\}$, a foliation on $U$ tangent to $\dfrac{\partial\Phi}{\partial x}$, so that $(\Phi \circ T \circ \Phi)_* \{dy_j=0\}=(\Phi \circ T)_*V$. The map $\Phi\circ T$ is holomorphic and injective on $\mathbb{D}^{k+1}\supset\overline U$, hence a biholomorphism from an open neighborhood $N$ of $\overline{U}$ onto its image, and $(\Phi\circ T)(N)\subset \mathbb{D}_{\frac12}^{k+1}\subset\mathbb{D}_2^{k+1}$. Therefore, if $(\Phi\circ T)_*V$ extended to a distribution $\GG$ on $\mathbb{D}_2^{k+1}$, the pullback $(\Phi\circ T)^*\GG$ would be a distribution on $N$ which coincides with $V$ on $U$ -- that is, an extension of $V$ to a neighborhood of $\overline U$, again contradicting Proposition \ref{Enunciado sobre Phi}(2).
\end{enumerate}
\end{proof}

The proposition above can be extended so as to preserve a prescribed set of foliations, as follows. This will give the cases $1\leq \ell\leq k+1$ of Theorem \ref{thm:A}.

\begin{prop}\label{ell=>1}
Let $\ell \in \{1, \ldots, k+1\}$, let $\Delta_2\subset \CC^{k+1}$ be an open set, $p \in \Delta_2$, let $C \subset \Delta_2$ be a smooth curve through $p$, and let $\mathcal{F}_1,\ldots, \mathcal{F}_{\ell}$ be codimension one foliations on $\Delta_2$, regular at $p$, such that:
\begin{enumerate}[(a)]
\item the tangent spaces $T_p \mathcal{F}_1, \ldots, T_p \mathcal{F}_{\ell}$ are in general position, i.e.\ their conormal directions at $p$ are linearly independent;
\item $C$ is transverse to $\mathcal{F}_1$ at $p$;
\item for $i=2,\dots,\ell$, the germ of $C$ at $p$ is contained in the leaf of $\FF_i$ through $p$.
\end{enumerate}
Then there is a biholomorphism $\Psi: \mathbb{D}^{k+1}\to \Delta$, where $p \in \Delta \subset \overline{\Delta}\subset \Delta_2$ and $\Psi(0)=p$, such that
\begin{enumerate}[(i)]
\item $\Psi^{-1}(C\cap \Delta)=A\cap \mathbb{D}^{k+1}$;
\item the foliation $\mathcal{F}_i|_{\Delta}$ is the image $\Psi_*\{dz_i = 0\}$, for all $i=1, \ldots, \ell$;
\item if $\ell \leq k$, then for any $j = \ell+1, \ldots, k+1$ the foliation $\Psi_*\{dz_j = 0\}$, defined on $\Delta$, does not extend to any neighborhood of $\overline{\Delta}$.
\end{enumerate}
\end{prop}
\begin{proof}
We first construct a ``rectifying'' chart. Let $u_1$ be a local holomorphic first integral of $\FF_1$ at $p$, with $u_1(p)=0$ and $du_1(p)\neq 0$. For $i=2,\dots,\ell$, let $u_i$ be a local holomorphic first integral of $\FF_i$ at $p$ vanishing on the local leaf through $p$; by hypothesis (c), $u_i$ vanishes on the germ of $C$ at $p$, so that $du_i(p)$ annihilates $T_pC$. Since $C$ is smooth at $p$, its conormal space $(T_pC)^{\perp}\subset T^*_p\CC^{k+1}$ has dimension $k$, and we may choose germs $u_{\ell+1},\dots,u_{k+1}$ vanishing on $C$ whose differentials at $p$ complete $du_2(p),\dots,du_{\ell}(p)$ to a basis of $(T_pC)^{\perp}$. By (a) and (b), $du_1(p)\notin (T_pC)^{\perp}$ and hence $du_1(p),\dots,du_{k+1}(p)$ is a basis of $T_p^*\CC^{k+1}$. Therefore $(u_1,\dots,u_{k+1})$ is a system of coordinates near $p$; moreover $\{u_2=\dots=u_{k+1}=0\}$ is a smooth curve through $p$ containing the germ of $C$, hence equal to it. After a linear rescaling we obtain a biholomorphism $Z: \mathbb{D}_2^{k+1}\to \Delta'\subset \Delta_2$, $Z(0)=p$, with $\overline{\Delta'}\subset\Delta_2$, satisfying $Z^{-1}(C\cap\Delta')=A\cap \mathbb{D}_2^{k+1}$ and $\FF_i|_{\Delta'}=Z_*\{dz_i=0\}$ for $i\leq \ell$.

If $\ell=k+1$ we simply take $\Psi:=Z|_{\mathbb{D}^{k+1}}$, and there is nothing more to prove. Assume then $\ell\leq k$, let $\Phi=\Phi_s$ be as in \eqref{phi} with $s=k+1-\ell$, and define
$$\widehat{\Phi}(z_1, \ldots , z_{k+1})= (z_1, \ldots, z_{\ell -1}, \Phi(z_{\ell}, \ldots,z_{k+1})), \qquad \Psi:=Z\circ \widehat{\Phi}.$$
The image of $\widehat\Phi$ has compact closure in $\mathbb{D}_2^{k+1}$, so $\overline{\Delta}\subset \Delta_2$ for $\Delta:=\Psi(\mathbb{D}^{k+1})$. Clearly $\widehat{\Phi}$ preserves the foliations $\{dz_i=0\}$ for $i=1, \ldots, \ell$ (for $i<\ell$ this is obvious, and for $i=\ell$ it follows from the fact that the first component of $\Phi$ is $z_{\ell}/2$), which gives (ii). Property (i) follows as in Proposition \ref{ell=0}: $\widehat\Phi^{-1}(A)=A$, because $\Phi(z_\ell,\dots,z_{k+1})=0$ forces $z_{\ell}=\dots=z_{k+1}=0$, while $\Phi$ maps the first axis of $\mathbb{D}\times\mathbb{D}^{s}$ to itself.

Finally, let us prove (iii). Suppose that for some $j \geq \ell +1$ the foliation $\Psi_{*}\{dz_j=0\}$ extends to a distribution $\GG$ on a neighborhood $\Omega$ of $\overline{\Delta}$; note that $\GG$ is automatically integrable by $(\star)$, since it is integrable on the open set $\Delta$ (we may assume $\Omega$ connected). Via the biholomorphism $Z$, we obtain a foliation on a neighborhood of the closure of $\widehat\Phi(\mathbb{D}^{k+1})$ extending $\widehat\Phi_*\{dz_j=0\}$. Consider the restriction to a slice $\{z_1=a_1, \ldots, z_{\ell-1}=a_{\ell-1}\}$, with $(a_1,\dots,a_{\ell-1})$ fixed: on this slice, which we identify with $\mathbb{D}_2^{s+1}$ with coordinates $(z_{\ell},\dots,z_{k+1})$, the extended foliation restricts to a (possibly singular) foliation defined near the closure of $\Phi(\mathbb{D}^{s+1})$, extending $\Phi_*\{dz_j=0\}$, which is tangent to $\dfrac{\partial \Phi}{\partial z_{\ell}}$. This contradicts Proposition \ref{Enunciado sobre Phi}(2).
\end{proof}

We now introduce the second ingredient. We say that a set $\mathcal{A} \subset \mathbb{C}$ is a \emph{$C^1$-arc} if there exists a $C^1$-homeomorphism $\varphi: (0,1) \rightarrow \mathcal{A}$  such that $\varphi'(t) \neq 0$ for all $t\in (0,1)$. Consider a continuous Jordan curve $\mathcal{C}$ that contains no $C^{1}$-arcs -- for instance, a suitably rescaled Koch snowflake --
and let $D$ denote its interior region. Suppose that $0\in D$.
Let $g: \mathbb{D} \to D$ be a uniformization map with $g(0)=0$.
By Carath\'eodory's theorem, $g$ extends to a homeomorphism
\[
g: \overline{\mathbb{D}} \longrightarrow \overline{D},
\]
which maps $\partial \mathbb{D}$ homeomorphically onto $\mathcal{C}$.
The following result will be particularly useful in what follows.

\begin{lemma}\label{Lema-principal}
The function $g:  \overline{\mathbb{D}}\rightarrow \overline{D}$ has the following properties:
\begin{enumerate}
\item There does not exist an open arc $A \subset \partial \mathbb{D}$ such that $\lim_{z \to a}g'(z)$ exists for all $a\in A$.
\item Let $\eta$ be a meromorphic $1$-form defined on a neighborhood of $\overline{D}$. Then $g^* \eta$ cannot be extended meromorphically to any neighborhood of $\overline{\mathbb D}$, unless $\eta$ is identically zero.
\end{enumerate}
\end{lemma}

\begin{proof}
Suppose there exists an open arc $A$ such that $\lim_{z \to a}g'(z)$ exists for all $a\in A$.

\noindent {\bf Claim: } For each $a \in A$ we have
$$\underset{ z \in \overline{\mathbb{D}}}{\lim_{z \to a}}\dfrac{g(z)-g(a)}{z-a} = \underset{ z \in \mathbb{D}}{\lim_{z \to a}}g'(z) .$$

In fact, writing $g(z) = u(z) + i v(z)$ and regarding $u,v$ as real functions on $\mathbb{D}\subset \mathbb{R}^2$, by the mean value theorem we have, for $z$ in the segment condition below,
$$
\dfrac{g(z)-g(a)}{z-a} = \dfrac{du(\alpha_z)\cdot (z-a)}{z-a} +i \dfrac{dv(\beta_z)\cdot (z-a)}{z-a},
$$
for some $\alpha_z, \, \beta_z$ in the open segment $(z,a)$. In particular $\alpha_z, \, \beta_z \in \mathbb{D}$ and $\alpha_z, \, \beta_z \to a$ when $z\to a$. By hypothesis
$\lim_{z \to a}g'(z) = A_0+iB_0 \in \mathbb{C}$. Since $g'(w) = u_x(w) + i v_x(w) = v_y (w) - iu_y(w)$ for all $w\in \mathbb{D}$, we have $u_x(w), v_y(w) \to A_0$ and $v_x(w), -u_y(w) \to B_0$ when $w\to a$. Therefore
$$
g(z)-g(a) = \left(\begin{matrix}
u_x(\alpha_z) & u_y(\alpha_z)\\
v_x(\beta_z) & v_y(\beta_z)
\end{matrix}\right) \cdot (z-a) \underset{z\to a}{\longrightarrow} \left(\begin{matrix}
A_0&-B_0\\
B_0 & A_0
\end{matrix}\right) \cdot (z-a) = (A_0+iB_0)(z-a),
$$
in the sense that the difference is $o(|z-a|)$, thus
$$
\frac{g(z)-g(a)}{z-a} \rightarrow A_0+iB_0
$$
and the claim follows.

By the claim, the function $g': \mathbb{D}\cup A \rightarrow \mathbb{C}$ is continuous. Let us prove the first item. Let $A=\{e^{it}: \, t \in(\theta_1, \theta_2)\}$ be a parametrization of the arc and $\gamma(t) = g(e^{it})$; by the claim, $\gamma$ is differentiable with $ \gamma'(t) = g'(e^{it})\cdot \underbrace{ie^{it}}_{\neq 0}$. Since $\gamma$ is a homeomorphism onto its image, we infer that $g'|_{A}\not\equiv 0$. By the continuity of $g'$ there exists an arc $A_0 \subset A$ such that $g'(a)\neq 0$ for all $a \in A_0$. Therefore $g(A_0)$ would be a $C^1$-subarc of $\mathcal{C}$, which leads to a contradiction.

For the second item, write $\eta=F(z) dz$ with $F$ meromorphic on a neighborhood of $\overline{D}$, $F\not\equiv 0$, and suppose that $g^*\eta=(F\circ g)g'\,dz$ extends meromorphically to a neighborhood of $\overline{\mathbb{D}}$. The set of points of $\mathcal{C}$ that are zeros or poles of $F$ is finite; hence there exists an open arc
\(A \subset \partial \mathbb{D}\) such that \(F ( g(z)) \neq 0,\infty\) for all \(z \in A\). On a neighborhood of $A$ in $\overline{\mathbb{D}}$, the extension of $(F\circ g)g'$ is finite up to shrinking $A$, so $g'=\frac{(F\circ g)g'}{F\circ g}$ admits a continuous limit at every point of $A$. This contradicts the first item.
\end{proof}

Fix now $r\in (0,1)$ and assume, as we may after rescaling, that $\overline{D}\subset \mathbb{D}_r$. Let $f: \mathbb{D} \rightarrow \mathbb{D}_r$ be holomorphic and injective with $f(0)=0$ (for instance $f(x)=rx/2$). Define
$$
h: \mathbb{D}\times \mathbb{D}^k \rightarrow \mathbb{C}^{k+1}, \,\, h(x, y_1, \ldots, y_k):=(f(x), g(y_1), \ldots, g(y_k)),
$$
and denote $U:= h( \mathbb{D}\times \mathbb{D}^k )=f(\mathbb{D})\times D^k$. Notice that $\overline{U}\subset \overline{\mathbb{D}}_r^{k+1}\subset\mathbb{D}^{k+1}$, and that for each of the $k+1$ coordinate foliations $\{dx=0\},\{dy_1=0\},\dots,\{dy_k=0\}$ the pullback $h^*$ is again the same coordinate foliation, which of course extends across the boundary $\{0\}\times\overline{\mathbb{D}}^k$. The next lemma says that these are the \emph{only} distributions whose pullback by $h$ survives, in the sense that if $\DD$ is any other distribution defined near $\overline U$, then $h^*\DD$ extends to no neighborhood of $\{0\}\times\overline{\mathbb{D}}^k$; equivalently, no distribution other than the coordinate ones pushes forward under $h$ to a distribution defined near that boundary.

\begin{lemma}\label{lema-sobre-h}
Let $\mathcal{D}$ be a codimension one distribution defined on a polydisc $V$ containing $\overline{U}$ and distinct from the $k+1$ distributions defined by the forms $dx, \, dy_1, \ldots , dy_k$. Then $h^* \mathcal{D}$ extends to no neighborhood of $\{0\}\times \overline{\mathbb{D}}^k$.
\end{lemma}
\begin{proof}
Write $\mathcal{D}=\{\omega = 0\}$ with $\omega=a\,dx + \sum_j b_j\, dy_j$ and $a, \, b_j \in \mathcal{O}(V)$. Suppose by contradiction that $h^* \mathcal{D}$ extends to a distribution on a neighborhood $\Omega$ of $\{0\}\times \overline{\mathbb{D}}^k$, and fix $\epsilon>0$ such that $\Omega_{\epsilon}:=\{(x,y)\in \mathbb{D}\times \mathbb{D}^k: \,\, |x|<\epsilon\}\subset \Omega$. Note that if $\widetilde\omega$ is a holomorphic $1$-form defining the extension on $\Omega_\epsilon$, then on $\Omega_\epsilon\cap (\mathbb{D}\times\mathbb{D}^k)$ we have $h^*\omega=m\,\widetilde\omega$ for some meromorphic function $m$; in particular, any quotient of $h^*\omega$ by one of its non zero coefficients extends meromorphically to $\Omega_\epsilon$, being equal to the corresponding quotient for $\widetilde\omega$.

We will show that at most one of the functions $a,b_1,\dots,b_k$ is not identically zero; this proves the lemma, for then $\DD$ is one of the coordinate distributions.

\noindent {\bf Case 1: $a \not\equiv 0$.} We show that $b_j \equiv 0$ for all $j$; by symmetry it suffices to prove $b_k\equiv 0$. We have
$$
h^*\omega= a(h) f'(x)dx + \sum_{j=1}^k b_j(h)g'(y_j)dy_j ,
$$
and, dividing by $a(h)f'(x)$ (which is $\not\equiv 0$ because $f$ is injective), the meromorphic $1$-form $\dfrac{h^*\omega}{a(h)f' (x)}$ extends meromorphically to $\Omega_{\epsilon}$, by the observation above. Given $(\tilde{x}, \tilde{y}_1, \ldots,\tilde{y}_{k-1})\in \mathbb{D}_{\epsilon}\times \mathbb{D}^{k-1}$ generic, consider the disc $D'=\{(\tilde{x}, \tilde{y}_1, \ldots, \tilde{y}_{k-1},y): y \in \mathbb{D} \}$. Since
$$\left.\frac{h^*\omega}{a(h)f'(x)}\right|_{D'}=\dfrac{1}{f'(\tilde{x})}\left( h|_{D'}\right)^*\left(\frac{\omega}{a} \Big|_{h(D')}\right),$$
we see that $\left( h|_{D'}\right)^*\left(\frac\omega{a} |_{h(D')}\right)$ extends meromorphically to a neighborhood of $\overline{D'}$. Now $h(D')$ is the disc $\{(f(\tilde x), g(\tilde y_1),\dots,g(\tilde y_{k-1}), w): w \in D\}$ and, in the coordinate $w$,
$$\frac{\omega}{a}\Big|_{h(D')}=\eta:=\frac{b_k}{a}\big(f(\tilde x), g(\tilde y_1),\dots,g(\tilde y_{k-1}), w\big)\, dw,$$
a meromorphic $1$-form on a neighborhood of $\overline{D}$ (recall that $\omega$ is holomorphic on $V\supset \overline{U}$). Moreover $\left( h|_{D'}\right)^*\eta = g^*\eta$ under the natural identifications. It follows from Lemma \ref{Lema-principal}(2) that $\eta \equiv 0$, that is, $b_k$ vanishes on $h(D')$. Letting the disc $D'$ vary we conclude that $b_k$ vanishes on the open set $U$, hence $b_k\equiv 0$ on $V$.

\noindent {\bf Case 2: $a\equiv 0$.} Then $b_{j_0} \not\equiv 0$ for some $j_0$, and we must show that $b_j\equiv 0$ for $j\neq j_0$. For simplicity assume $j_0=1$ and let us show that $b_k\equiv 0$ (the other cases are analogous). This time
$$
h^*\omega =  \sum_{j=1}^k b_j(h)g'(y_j) dy_j,
$$
and, as in the previous case, $\dfrac{h^*\omega}{b_1 (h)g'(y_1)}$ extends meromorphically to $\Omega_{\epsilon}$. Thus, if $D'$ is any disc as above, we have
$$
\left.\frac{h^*\omega}{b_1(h)g'(y_1)}\right|_{D'} = \frac{1}{g'(\tilde{y}_1)}\left( h|_{D'}\right)^*\left(\frac{1}{b_1}\omega \Big|_{h(D')}\right).
$$
We conclude as before, using Lemma \ref{Lema-principal}(2), that $\frac{1}{b_1}\omega |_{h(D')}\equiv 0$, which implies that $b_k \equiv 0$.
\end{proof}

\section{Construction of special manifolds}\label{sec:gluing}

In this section we use a gluing procedure to construct $(k+1)$-dimensional manifolds endowed with finitely many foliations. Fix
\[
h(x,y) = h(x, y_1, \ldots, y_k) = \bigl(f(x), g(y_1), \ldots, g(y_k)\bigr)
\]
as in the previous section, and recall that $\overline{U}=\overline{h(\mathbb{D} \times \mathbb{D}^k)}\subset \mathbb{D}_r^{k+1}$.

Let \(C \subset \mathbb{P}^{k+1}\) be a connected projective curve, and fix a smooth point \(p \in C\). Consider a neighborhood \(\tilde{V}\) of \(p\) with a chart \(\tilde{\psi}: \tilde{V} \to V\), where $V$ is a polydisc containing $\overline U$, such that \(\tilde{\psi}(p) = 0\) and \(\tilde{\psi}(C\cap \tilde V) = \{ y = 0 \}\).

Likewise, let $\pi: Bl_0\,\CC^{k+1}\to \CC^{k+1}$ be the blow-up of the origin, with exceptional divisor \(E =\pi^{-1}(0)\simeq \Pp^k\), and fix a point \(q \in E\). Take a neighborhood \(\tilde{W}\) of \(q\) with a chart \(\tilde{\varphi}: \tilde{W} \to \mathbb{D} \times \mathbb{D}^k\) such that \(\tilde{\varphi}(q) = 0\) and \(\tilde{\varphi}(E\cap \tilde W) = \{ x = 0 \}\). We assume that $\tilde\varphi$ is the restriction of a biholomorphism defined on a neighborhood of the closure of $\tilde W$; in particular, $\tilde{\varphi}$ extends as a homeomorphism between the closure of $\tilde{W}$ and $\overline{\mathbb{D}}\times \overline{\mathbb{D}}^k$. A concrete choice, which we fix from now on, is the following: let $(t,u)=(t,u_2,\dots,u_{k+1})$ be the standard chart of $Bl_0\,\CC^{k+1}$ in which $\pi(t,u)=(t,tu_2,\dots,tu_{k+1})$ and $E=\{t=0\}$, take $q$ the point with coordinates $(t,u)=0$, $\tilde W:=\{(t,u)\in \mathbb{D}\times\mathbb{D}^k\}$ and $\tilde\varphi:=(t,u)|_{\tilde W}$, with $x=t$ and $y_j=u_{j+1}$.

Define
\[
\tilde{h} := \tilde{\psi}^{-1} \circ h \circ \tilde{\varphi}: \tilde{W} \to \tilde{\psi}^{-1}(U)\subset \tilde V,
\]
a biholomorphism onto its image. We choose neighborhoods \(T_C\supset C\) in $\Pp^{k+1}$ and \(T_E\supset E\) in $Bl_0\,\CC^{k+1}$, and consider the space
\[
S_0 := T_C \sqcup_{\tilde{h}} T_E,
\]
the quotient of the disjoint union $T_C\sqcup T_E$ by the identification of $w\in \tilde W$ with $\tilde h(w)\in T_C$. The construction is summarized in Figure \ref{fig:gluing}. In general, \(S_0\) is a non-Hausdorff manifold. To ensure that \(S_0\) is Hausdorff, we specify \(T_C\) and \(T_E\) as follows.

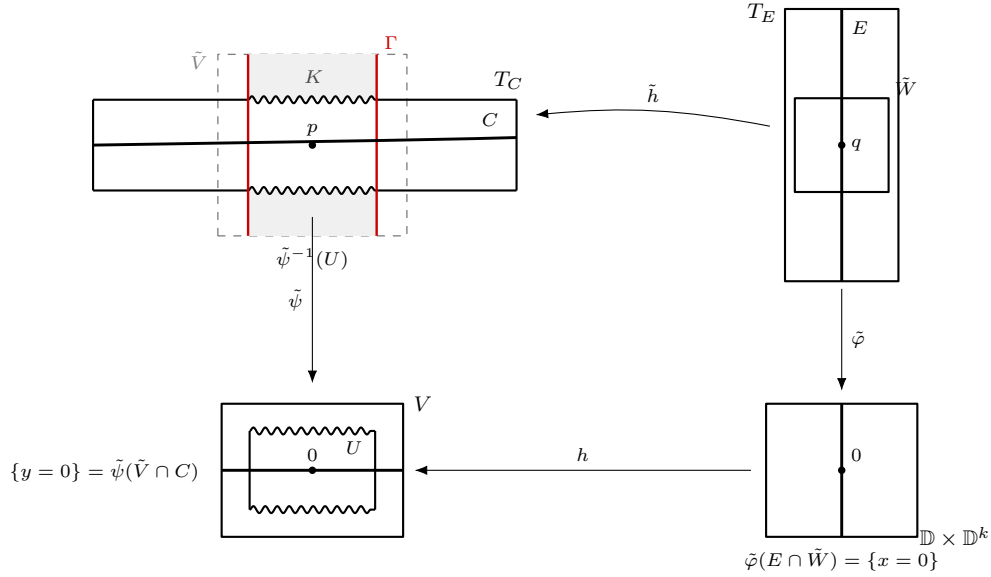
\begin{figure}[ht]
\centering
\begin{tikzpicture}[>=Latex,line join=round,scale=1.0]
\tikzset{
  corr/.style={decorate,decoration={snake,amplitude=1.3pt,segment length=4.5pt}},
  redline/.style={red!80!black,line width=0.9pt},
  crv/.style={line width=1.1pt},
  box/.style={line width=0.8pt},
  chart/.style={dashed,line width=0.5pt,gray},
  lbl/.style={font=\footnotesize},
  slbl/.style={font=\scriptsize}
}

\begin{scope}[shift={(-0.7,4.3)}]
  \draw[chart] (-1.25,-1.2) rectangle (1.25,1.2);
  \fill[gray!12] (-0.85,0.6) rectangle (0.85,1.2);
  \fill[gray!12] (-0.85,-1.2) rectangle (0.85,-0.6);
  \draw[box] (-2.9,0.6) -- (-0.85,0.6);
  \draw[box,corr] (-0.85,0.6) -- (0.85,0.6);
  \draw[box] (0.85,0.6) -- (2.7,0.6);
  \draw[box] (-2.9,-0.6) -- (-0.85,-0.6);
  \draw[box,corr] (-0.85,-0.6) -- (0.85,-0.6);
  \draw[box] (0.85,-0.6) -- (2.7,-0.6);
  \draw[box] (-2.9,-0.6) -- (-2.9,0.6);
  \draw[box] (2.7,-0.6) -- (2.7,0.6);
  \draw[redline] (-0.85,-1.2) -- (-0.85,1.2);
  \draw[redline] (0.85,-1.2) -- (0.85,1.2);
  \draw[crv] (-2.9,0) .. controls (-0.5,0.04) and (1.2,0.06) .. (2.7,0.1);
  \fill (0,0) circle (1.4pt) node[above,slbl] {$p$};
  \node[slbl] at (2.35,0.34) {$C$};
  \node[lbl] at (2.6,0.86) {$T_C$};
  \node[gray,slbl] at (-1.5,1.12) {$\tilde V$};
  \node[gray!55!black,slbl] at (0.02,0.9) {$K$};
  \node[slbl] at (0,-1.5) {$\tilde\psi^{-1}(U)$};
  \node[red!80!black,slbl] at (1.04,1.36) {$\Gamma$};
\end{scope}

\begin{scope}[shift={(6.3,4.3)}]
  \draw[box] (-0.75,-1.8) rectangle (0.75,1.8);            
  \draw[crv] (0,-1.8) -- (0,1.8);                           
  \draw[box] (-0.62,-0.62) rectangle (0.62,0.62);           
  \fill (0,0) circle (1.4pt) node[right,slbl] {$q$};
  \node[slbl] at (0.24,1.55) {$E$};
  \node[lbl] at (-1.05,1.75) {$T_E$};
  \node[slbl] at (0.85,0.78) {$\tilde W$};
\end{scope}

\begin{scope}[shift={(-0.7,0)}]
  \draw[box] (-1.2,-0.88) rectangle (1.2,0.88);            
  \draw[box,corr] (-0.83,0.52) -- (0.83,0.52);
  \draw[box,corr] (-0.83,-0.52) -- (0.83,-0.52);
  \draw[box] (-0.83,-0.52) -- (-0.83,0.52);
  \draw[box] (0.83,-0.52) -- (0.83,0.52);
  \draw[crv] (-1.2,0) -- (1.2,0);                           
  \fill (0,0) circle (1.4pt) node[above,slbl] {$0$};
  \node[lbl] at (1.46,0.88) {$V$};
  \node[slbl] at (0.55,0.30) {$U$};
  \node[slbl] at (-2.75,0) {$\{y=0\}=\tilde\psi(\tilde V\cap C)$};
\end{scope}

\begin{scope}[shift={(6.3,0)}]
  \draw[box] (-1.0,-0.88) rectangle (1.0,0.88);
  \draw[crv] (0,-0.88) -- (0,0.88);                         
  \fill (0,0) circle (1.4pt) node[above right,slbl] {$0$};
  \node[lbl] at (1.5,-0.88) {$\mathbb{D}\times\mathbb{D}^k$};
  \node[slbl] at (0,-1.18) {$\tilde\varphi(E\cap\tilde W)=\{x=0\}$};
\end{scope}

\draw[->] (-0.7,3.35) -- (-0.7,1.15) node[midway,left,slbl] {$\tilde\psi$};
\draw[->] (6.3,2.4) -- (6.3,1.05) node[midway,right,slbl] {$\tilde\varphi$};

\draw[->] (5.35,4.55) to[bend right=8] node[midway,above,slbl] {$\tilde h$} (2.25,4.7);
\draw[->] (5.1,0) -- (0.65,0) node[midway,above,slbl] {$h$};
\end{tikzpicture}
\caption{The gluing $S_0:=T_C\sqcup_{\tilde h}T_E$. Downstairs, the map $h(x,y)=(f(x),g(y_1),\dots,g(y_k))$ sends the smooth box $\mathbb{D}\times\mathbb{D}^k$ onto the ``horn'' $U=f(\mathbb{D})\times D^k\subset V$: since $g$ uniformizes the Jordan domain $D$, whose boundary contains no $C^1$-arc, the top and bottom sides of $U$ become wavy. Upstairs, $\tilde h=\tilde\psi^{-1}\circ h\circ\tilde\varphi$ carries the small neighborhood $\tilde W$ of $q\in E$ onto $\tilde\psi^{-1}(U)$, gluing the neighborhood $T_E$ of the \emph{whole} exceptional divisor $E\simeq\mathbb{P}^k$ (drawn as a vertical line) to the neighborhood $T_C$ of the curve $C$; thus $E$ and $C$ meet transversely at the single point $p=q$. The chart box $\tilde V$ is a larger polydisc containing $\tilde\psi^{-1}(U)$; the strip $\Gamma$, between the red lines, is the horn prolonged vertically up to $\tilde V$, and $T_C$ excludes the closure of $K=\Gamma\setminus\tilde\psi^{-1}(U)$, so its boundary runs \emph{along} the wavy sides of the horn.}
\label{fig:gluing}
\end{figure}  On the one hand, for \(T_E\) we take
\[
T_E:=\pi^{-1}(\mathbb{D}_{\varepsilon}^{k+1})\cup \tilde W \qquad (0<\varepsilon<1),
\]
which is a connected neighborhood of $E$ satisfying
\begin{equation}\label{eq:TEconditions}
\tilde{W} \subset T_E, \qquad
T_E \cap \tilde{\varphi}^{-1}(\partial \mathbb{D} \times \overline{\mathbb{D}}^k) = \varnothing,
\qquad \pi(T_E)\subset \mathbb{D}^{k+1};
\end{equation}
indeed, a point of $\tilde\varphi^{-1}(\partial\mathbb{D}\times\overline{\mathbb{D}}^k)$ has $|t|=1$, hence its image by $\pi$ has first coordinate of modulus $1>\varepsilon$, and $\pi(\tilde W)\subset \mathbb{D}^{k+1}$ because $|t|,|tu_j|<1$ on $\tilde W$.

On the other hand, let \(\Gamma \subset \tilde{V}\) be defined, in the chart \((x,y)\) given by \(\tilde{\psi}\), as  $\Gamma := \{(x,y)\in V : x \in f(\mathbb{D}) \}$; observe that \(\tilde{h}(\tilde{W}) =\tilde\psi^{-1}(U)\subset \Gamma\) and that $\Gamma$ decomposes as the disjoint union of $\tilde h(\tilde W)$ and $K:=\tilde\psi^{-1}\big(f(\mathbb{D})\times (V_2\setminus D^k)\big)$, where $V_2$ denotes the product of the last $k$ factors of the polydisc $V$. Since $0\in D$ and $D$ is open, the curve $C$ does not meet the closure of $K$, and we may define
$$T_C:= (N\cup \tilde V)\setminus \overline{K},$$
where $N$ is any neighborhood of $C$ in $\Pp^{k+1}$. Then $T_C$ is an open neighborhood of $C$ satisfying $T_C \cap \Gamma = \tilde{h}(\tilde{W})$.

\medskip
\noindent{\bf The manifold $S_0$ is Hausdorff.} Suppose by contradiction that $a\in T_C$ and $b \in T_E$ define distinct non-separated points of $S_0$. Then there is a sequence $(w_n)\subset \tilde W$ with $w_n \to b$ in $T_E$ and $\tilde h(w_n)\to a$ in $T_C$; moreover $b\notin \tilde W$, for otherwise $a=\tilde h(b)$ would be identified with $b$. Write $\tilde\varphi(w_n)=(x_n,y_n)$; up to a subsequence, $(x_n,y_n)\to (\bar x, \bar y)\in \partial(\mathbb{D}\times\mathbb{D}^k)$, and $b=\tilde\varphi^{-1}(\bar x,\bar y)$ (extended homeomorphism). If $\bar x \in \partial \mathbb{D}$, then $b \in \tilde\varphi^{-1}(\partial\mathbb{D}\times\overline{\mathbb{D}}^k)$, contradicting \eqref{eq:TEconditions}. Hence $\bar x \in \mathbb{D}$ and $|\bar y_j|=1$ for some $j$. Then
$$\tilde\psi(a)=\lim h(x_n,y_n) = (f(\bar x), g(\bar y_1),\dots,g(\bar y_k)),$$
whose first coordinate lies in $f(\mathbb{D})$ and whose $j$-th ``vertical'' coordinate $g(\bar y_j)$ lies in $\mathcal{C}=\partial D$. The first condition gives $a \in \Gamma\cap T_C=\tilde h(\tilde W)=\tilde\psi^{-1}(U)$, so all vertical coordinates of $\tilde\psi(a)$ lie in the open set $D$ -- contradicting $g(\bar y_j)\in \partial D$. This proves the claim. (Non-separated pairs with both points in $T_C$, or both in $T_E$, do not occur, since $\tilde h$ is injective.)

\medskip
With these choices, we obtain a complex manifold
\[
S_0 := T_C \sqcup_{\tilde{h}} T_E
\]
containing copies of both neighborhoods via the natural inclusions
\(T_C \hookrightarrow S_0\) and \(T_E \hookrightarrow S_0\). Observe that, inside $S_0$, the curve $C$ and the divisor $E$ intersect transversely at the single point $p$: indeed, $q\in E$ is identified with $\tilde h(q)=\tilde\psi^{-1}(h(0))=p$, and near this point $E$ is contained in $\{x=0\}$ while $C=\{y=0\}$, in the chart $\tilde\psi$.

We state the main result of this section.

\begin{thm}\label{thm:coordinate}
Let \(\mathcal{D}\) be a codimension one distribution on \(S_0\).
Then the restriction \(\mathcal{D}|_{T_C}\) extends to a distribution $\widehat{\DD}$ on \(\mathbb{P}^{k+1}\), and \(\tilde{\psi}_* (\widehat{\DD}|_{\tilde V})\) is generated by one of the forms in the set
\(\{ dx, dy_1, \ldots, dy_k \}\).
In particular, \(\mathcal{D}\) is integrable.
\end{thm}
\begin{proof}
For the first statement, observe that \(\mathcal{D}|_{T_C}\) defines a distribution on a
neighborhood of a positive-dimensional subvariety of the projective space
\(\mathbb{P}^{k+1}\).
By \cite[Theorem 3.1]{Ros}, such a distribution extends to a global distribution $\widehat\DD$ on
\(\mathbb{P}^{k+1}\), which agrees with $\DD$ on $T_C$.

For the second statement, suppose that
\(\tilde{\psi}_* (\widehat{\DD}|_{\tilde V})\), which is a distribution on the whole polydisc $V\supset \overline U$, is not generated by any of the forms
\(dx, dy_1, \ldots, dy_k\).
By Lemma~\ref{lema-sobre-h}, the pullback $h^*\big(\tilde\psi_*(\widehat\DD|_{\tilde V})\big)$ cannot
extend to any neighborhood of \(\{0\} \times \overline{\mathbb{D}}^k\).
On the other hand, $\DD$ is also defined on $T_E$, which contains the compact set $\tilde\varphi^{-1}(\{0\}\times \overline{\mathbb{D}}^k)\subset E$; expressing $\DD|_{T_E}$ in the chart $(t,u)$ -- which, by our choice of $\tilde\varphi$, is defined on a neighborhood of $\overline{\mathbb{D}}\times\overline{\mathbb{D}}^k$ -- we obtain a distribution on a neighborhood of $\{0\}\times\overline{\mathbb{D}}^k$ in $\CC^{k+1}$ whose restriction to $\mathbb{D}\times\mathbb{D}^k$ is exactly $\tilde\varphi_*(\DD|_{\tilde W})=h^{*}\big(\tilde\psi_*(\DD|_{\tilde h(\tilde W)})\big)=h^*\big(\tilde\psi_*(\widehat\DD|_{\tilde V})\big)|_{\mathbb{D}\times\mathbb{D}^k}$, where we used that $\tilde h=\tilde\psi^{-1}\circ h\circ\tilde\varphi$ identifies the two restrictions of $\DD$. This is the desired contradiction.
Therefore \(\tilde{\psi}_* (\widehat\DD|_{\tilde V})\) is generated by one of the coordinate forms; these being integrable, $\DD$ is integrable on a non empty open subset of the connected manifold $S_0$, hence on all of $S_0$ by $(\star)$.
\end{proof}

\begin{remark}\label{rem:count}
Recall that \( \mathrm{Distr}(S) \) denotes the set of codimension one distributions on a complex manifold $S$, and \( \mathrm{Fol}(S) \) the set of codimension one foliations. The construction above, suitably specified, provides examples of manifolds for which \( \mathrm{Distr}(S) \) is finite; this will be carried out in Section \ref{sec:proofA}, proving Theorem \ref{thm:A}. Moreover, in these examples we have
\[
\#\mathrm{Distr}(S) \leq \dim S \quad \text{and} \quad \mathrm{Distr}(S) = \mathrm{Fol}(S).
\]
These examples raise the following natural questions.
\end{remark}
\begin{question}\label{quest}
Let $S$ be a connected complex manifold.
\begin{enumerate}
    \item If \( \mathrm{Distr}(S) \) is finite, is it true that \( \#\mathrm{Distr}(S) \leq \dim S \)?
    \item If \( \mathrm{Distr}(S) \) is finite, is it necessarily true that \( \mathrm{Distr}(S) = \mathrm{Fol}(S) \)?
    \item If \( \mathrm{Distr}(S) \) is infinite, can \( \mathrm{Fol}(S) \) still be finite and non empty?
    \item If \( \mathrm{Fol}(S) \) is finite, is it true that \( \#\mathrm{Fol}(S) \leq \dim S \)?
\end{enumerate}
\end{question}

Proposition \ref{prop:genposbound} below answers question (1) affirmatively, for an arbitrary complex manifold. Questions (2), (3) and (4) remain open. Let us stress that these questions are only meaningful for non-algebraic manifolds: as observed in Remark \ref{rem:BP}, a projective manifold of dimension at least two always carries infinitely many codimension one foliations.

\begin{remark}\label{rem:BP}
It is instructive to contrast Theorem \ref{thm:A} with the projective situation, where the phenomenon we are after simply cannot occur. Indeed, if $X$ is a projective manifold with $\dim X=n\geq 2$, then $\CC(X)$ has transcendence degree $n\geq 2$ over $\CC$, so we may pick algebraically independent $f,g\in \CC(X)$; by Proposition \ref{prop:pencil}, the pencil $\{d(f+tg)=0\}_{t\in\CC}$ already consists of pairwise distinct foliations. Hence
$$\Fol(X)\ \text{ is infinite for every projective manifold } X \text{ with } \dim X\geq 2,$$
and questions (1)--(4) of \ref{quest} are empty in the projective category: finiteness of $\Fol$ or $\Distr$ is a genuinely non-algebraic, non-compact phenomenon. This is what makes germs of neighborhoods the natural setting for Theorem \ref{thm:A}, and it explains why our manifolds $S$ must have few meromorphic functions (Theorem \ref{thm:B}): by the same pencil argument, two functionally independent meromorphic functions would already force infinitely many foliations.

The meaningful projective question is therefore not how many foliations $X$ carries, but how many codimension one foliations can contain a \emph{fixed} foliation $\GG$; equivalently, when the intersection $\FF_1\cap\dots\cap\FF_m$ of $m$ pairwise distinct codimension one foliations -- the foliation whose tangent sheaf is $T\FF_1\cap\dots\cap T\FF_m$, of codimension at most $\min(m,n)$ -- has smaller codimension than expected. Such \emph{unlikely intersections} are the subject of the work of Barbosa and Pereira \cite{BP}, whose main finiteness statement is \cite[Theorem D]{BP}:
\begin{center}
\emph{if a foliation $\GG$ of codimension $q$ is contained in at least $q+1$ pairwise distinct codimension one foliations, then it is contained in infinitely many of them.}
\end{center}
Note that the naive count is sharp: a codimension $q$ foliation is contained in at most $q$ codimension one foliations, unless it is contained in infinitely many. The bound cannot be improved, as Section \ref{sec:pereira} shows: there we exhibit a codimension two foliation $\GG$ on $\Pp^3$ contained in exactly two codimension one foliations.

Let us briefly indicate what lies behind \cite[Theorem D]{BP}, since it illuminates the difference with our non compact examples. The proof analyzes the transverse structure of $\GG$ according to the transcendence degree $\delta$ of the field $\CC(X/\GG)$ of rational first integrals of $\GG$. If $\delta\geq 2$, pencils of first integrals immediately produce infinitely many foliations containing $\GG$ -- this is Proposition \ref{prop:pencil} again, and it is also the (degenerate) reason why $\Fol(X)$ is always infinite, since the foliation by points has $\CC(X/\GG)=\CC(X)$. The substance of the theorem is thus in the range $\delta\leq 1$: there the abundance of foliations containing $\GG$ forces on it a rich transverse structure -- transversely affine, projective or Lie -- and the symmetries of this structure generate pencils, or even conics, of integrable $1$-forms containing $\GG$ \cite[Section 8]{BP}. The case $q=2$ rests on the study of pencils of integrable $1$-forms, which goes back to Cerveau's work on $\mathbb{P}^3$ \cite{Cer} and is extended to arbitrary projective manifolds in \cite[Theorem A]{BP}. A further reduction, \cite[Lemma 8.1]{BP}, allows one to replace an arbitrary family realizing an unlikely intersection by a minimal one, for which all the $q$-fold intersections coincide.
\end{remark}

\begin{remark}\label{rem:genpos}
Recall that a finite set of points $\{p_0, p_1, \dots, p_m\} \subset \mathbb{P}^d$ is said to be in \emph{general position} if no subset of $j+1$ of them, with $j \le d$, lies in a projective subspace of dimension smaller than $j$. It is a classical fact that any set of $d+2$ points of $\mathbb{P}^d$ in general position determines a unique projective automorphism
\(\varphi \in \mathrm{PGL}_{d+1}\)
sending these points to the standard projective frame $[1\!:\!0\!:\!\dots\!:\!0],\; [0\!:\!1\!:\!\dots\!:\!0],\; \dots,\;
[0\!:\!\dots\!:\!1],\; [1\!:\!1\!:\!\dots\!:\!1]$. This is the geometry underlying the following proposition.
\end{remark}

\begin{prop}\label{prop:genposbound}
Let $S$ be a connected complex manifold. If $\mathrm{Distr}(S)$ is finite, then $\#\,\mathrm{Distr}(S) \le \dim S$.
Moreover, the distributions are in general position: if $\mathrm{Distr}(S)=\{\DD_1,\dots,\DD_m\}$, then for a generic point $p \in S$ the projective classes $[\mathcal{D}_1(p)], \ldots, [\mathcal{D}_m(p)]$, viewed as points of $\mathbb{P}T^*_pS$, are in general position.
\end{prop}

\begin{proof}
Here $[\DD_i(p)]\in \Pp T^*_pS$ denotes the conormal direction of $\DD_i$ at a generic point $p$. We claim that for any subset  $\{\mathcal{D}_{i_0}, \ldots, \mathcal{D}_{i_d}\} \subset \Distr(S)$ of $d+1$ pairwise distinct distributions and generic $p \in S$, the points $[\mathcal{D}_{i_0}(p)], \ldots, [\mathcal{D}_{i_d}(p)] \in \mathbb{P}T^*_pS$ span a projective subspace of dimension $d$, i.e.\ the corresponding conormal directions are linearly independent. The claim proves both statements: general position follows by definition, and if $\#\Distr(S)\geq \dim S+1=:m'+1$, a subset with $m'+1$ elements would span a $\mathbb{P}^{m'}$ inside $\Pp T^*_pS\simeq \Pp^{m'-1}$, which is absurd.

Suppose, by contradiction, that the claim fails, and let $d$ be the minimal integer for which it fails, say for the subset $\{\DD_0,\dots,\DD_d\}$ (after relabeling). Since the distributions are pairwise distinct we necessarily have $d\geq 2$: for $d=1$, the failure would mean $[\DD_0(p)]=[\DD_1(p)]$ for generic $p$, i.e.\ $\DD_0=\DD_1$. Choose local holomorphic $1$-forms $\omega_0,\dots,\omega_d$ generating $\DD_0,\dots,\DD_d$ on some open set. By minimality of $d$, the forms $\omega_0(p),\dots,\omega_{d-1}(p)$ are linearly independent for generic $p$, while
$$\omega_d = \lambda_0\,\omega_0+\dots+\lambda_{d-1}\,\omega_{d-1}$$
for uniquely determined meromorphic functions $\lambda_0,\dots,\lambda_{d-1}$ (solve the linear system pointwise; the solutions are meromorphic by Cramer's rule). Moreover $\lambda_i\not\equiv 0$ for every $i$: if $\lambda_i\equiv 0$, then the $d$ distributions $\{\DD_j\}_{j\neq i, j<d}\cup\{\DD_d\}$ would have linearly dependent conormals at the generic point, contradicting the minimality of $d$.

Now, for $\mu=[\mu_0:\dots:\mu_{d-1}]\in \Pp^{d-1}$, consider the meromorphic $1$-form
$$\omega_{\mu}:=\mu_0\lambda_0\,\omega_0+\dots+\mu_{d-1}\lambda_{d-1}\,\omega_{d-1}$$
and let $\DD_{\mu}$ be the codimension one distribution it defines (after clearing denominators and saturating). The functions $\lambda_i$ play no role in the pointwise linear algebra; their sole purpose is to make the above combination \emph{globally} meaningful. Indeed, the naive combination $\mu_0\omega_0+\dots+\mu_{d-1}\omega_{d-1}$ would not do: the generators $\omega_i$ are only defined up to multiplication by nowhere vanishing holomorphic functions $g_i$, which may be chosen independently of each other, and such a combination is turned into $\sum_i\mu_ig_i\omega_i$, which is not proportional to it. The coefficients $\lambda_i$ normalize the generators so that all the terms acquire one and the same factor: replacing $\omega_i$ by $g_i\omega_i$ and $\omega_d$ by $g_d\omega_d$ changes $\lambda_i$ into $\lambda_ig_d/g_i$, whence
$$\lambda_i\omega_i\ \longmapsto\ \frac{\lambda_ig_d}{g_i}(g_i\omega_i)=g_d\,\lambda_i\omega_i \qquad \text{and} \qquad \omega_\mu\longmapsto g_d\,\omega_\mu.$$
Equivalently, the $1$-forms $\eta_i:=\lambda_i\omega_i$ still generate $\DD_i$ (because $\lambda_i\not\equiv0$) and are normalized by the condition $\eta_0+\dots+\eta_{d-1}=\omega_d$, which determines them up to a common factor. Therefore the locally defined distributions glue to a global $\DD_{\mu}\in \Distr(S)$. Finally, for generic $p$ the vectors $\omega_0(p),\dots,\omega_{d-1}(p)$ are linearly independent and $\lambda_i(p)\neq 0$, so distinct values of $\mu$ give distinct conormal directions $[\omega_{\mu}(p)]$, hence distinct distributions. We have produced an injection $\mathbb{P}^{d-1}  \hookrightarrow \Distr(S)$ with $d-1\geq 1$, so $\Distr(S)$ is infinite -- a contradiction.
\end{proof}

\begin{remark}
The distributions $\DD_\mu$ produced in the proof have no reason to be integrable, even when the $\DD_i$ are; this is why the argument does not answer question (4) of \ref{quest}. For pencils ($d=2$) of \emph{integrable} forms on projective manifolds, integrability of the whole family is precisely the situation studied in \cite{Cer, BP}, cf.\ Remark \ref{rem:BP}.
\end{remark}

\section{Proof of Theorem \ref{thm:A}}\label{sec:proofA}

We keep the notation of Section \ref{sec:gluing}: $\pi: Bl_0\,\CC^{k+1}\to \CC^{k+1}$ is the blow-up of the origin, $E=\pi^{-1}(0)$, and $(t,u)$ is the standard chart with $\pi(t,u)=(t,tu_2,\dots,tu_{k+1})$.

The three auxiliary lemmas of this section are standard, and will certainly be clear to specialists; we claim no originality for them and include complete proofs only in order to keep the paper self-contained and readable by non-specialists.

We first record the elementary but crucial fact that all analytic objects on a neighborhood of $E$ come from the ball, and that the coordinate foliations of the chart $(t,u)$ are globally defined along $E$.

\begin{lemma}\label{lem:Eside}
Let $T_E$ be a connected neighborhood of $E$ in $Bl_0\,\CC^{k+1}$.
\begin{enumerate}[(1)]
\item For every $F\in \MM(T_E)$ there exist $\varepsilon>0$ with $\pi^{-1}(\mathbb{D}^{k+1}_{\varepsilon})\subset T_E$ and $G\in \MM(\mathbb{D}^{k+1}_{\varepsilon})$ such that $F=G\circ \pi$ on $\pi^{-1}(\mathbb{D}^{k+1}_{\varepsilon})$. The analogous statement holds for codimension one distributions and foliations.
\item Each of the $k+1$ coordinate foliations $\{dt=0\}, \{du_2=0\},\dots,\{du_{k+1}=0\}$ of the chart $(t,u)$ is the restriction of a singular holomorphic foliation defined on all of $Bl_0\,\CC^{k+1}$, namely
$$\{dt=0\}=\pi^*\{dz_1=0\}, \qquad \{du_j=0\}=\pi^*\{z_1dz_j-z_jdz_1=0\}.$$
\end{enumerate}
\end{lemma}
\begin{proof}
(1) Since $E$ is compact, $\pi^{-1}(\overline{\mathbb{D}}_{\varepsilon}^{k+1})\subset T_E$ for some $\varepsilon>0$. As $\pi$ is a biholomorphism off $E$, the function $F$ induces a meromorphic function on $\mathbb{D}^{k+1}_{\varepsilon}\setminus\{0\}$, which extends to $\mathbb{D}^{k+1}_{\varepsilon}$ by Levi's extension theorem, since $k+1\geq 2$. For a distribution $\DD$ on $T_E$, take a holomorphic $1$-form $\omega=\sum a_i\,dz_i$ defining its direct image on $\mathbb{D}^{k+1}_{\varepsilon}\setminus \{0\}$: the coefficients $a_i$ extend holomorphically through the origin by Hartogs' theorem, and the extended form defines a distribution $\DD_0$ with $\pi^*\DD_0=\DD$ near $E$; integrability is preserved.

(2) In the chart, $\pi^*dz_1=dt$; and since $z_1=t$, $z_j=tu_j$,
$$\pi^*(z_1dz_j-z_jdz_1)= t(u_jdt+tdu_j)-tu_j\,dt=t^2du_j,$$
which, after dividing by $t^2$, generates $\{du_j=0\}$. The forms $z_1dz_j-z_jdz_1$ are integrable with singular set of codimension two (they define the pencils of hyperplanes $\{z_j=cz_1\}$), so their pullbacks define singular foliations on all of $Bl_0\,\CC^{k+1}$.
\end{proof}

Next we construct, on $\Pp^{k+1}$, foliations adapted to the curve $C$ and to the point $p$, to which Proposition \ref{ell=>1} will be applied.
\begin{lemma}\label{lem:adapted}
Let $C\subset \Pp^{k+1}$ be a connected projective curve, $p\in C$ a smooth point and $\ell \in\{1,\dots,k+1\}$. There exist codimension one singular holomorphic foliations $\FF_1,\dots,\FF_{\ell}$ on $\Pp^{k+1}$, all regular at $p$, satisfying hypotheses (a), (b), (c) of Proposition \ref{ell=>1} (with respect to the germ of $C$ at $p$). Moreover, $\FF_1$ can be chosen so that either:
\begin{enumerate}[(i)]
\item $\FF_1$ is the pencil of hyperplanes containing a fixed codimension two linear subspace; in this case $\FF_1$ admits the rational first integral $R_0=H/H'$, a quotient of two linear forms, and every rational first integral of $\FF_1$ is of the form $\phi\circ R_0$ with $\phi\in \CC(w)$; or
\item $\FF_1$ admits no non-constant rational first integral.
\end{enumerate}
\end{lemma}
\begin{proof}
We first construct $\FF_2,\dots,\FF_{\ell}$; note that $\ell-1\leq k$. Let $\II_C$ be the ideal sheaf of $C$, let $\mathfrak{m}_p\subset \OO_{\Pp^{k+1}}$ be the ideal sheaf of the point $p$ and let $\kappa(p)=\OO_{\Pp^{k+1},p}/\mathfrak{m}_p\simeq \CC$ be the residue field at $p$.

We shall record sections of $\II_C(m)$ not by their values at $p$ -- which vanish, since such sections vanish along $C\ni p$ -- but by their \emph{differentials} at $p$. This is expressed by the fiber
$$\II_C\otimes \kappa(p)\;=\;\II_{C,p}/\mathfrak{m}_p\II_{C,p},$$
a finite dimensional $\CC$-vector space. Indeed, as $\II_{C,p}\subset \mathfrak{m}_p$, the Leibniz rule gives $d(fs)(p)=f(p)\,ds(p)+s(p)\,df(p)=0$ for $f\in\mathfrak{m}_p$ and $s\in \II_{C,p}$, so that $s\mapsto ds(p)$ vanishes on $\mathfrak{m}_p\II_{C,p}$ and induces a linear map
$$\II_C\otimes \kappa(p)\longrightarrow T^*_p\Pp^{k+1},\qquad s\longmapsto ds(p),$$
whose image is the conormal space $(T_pC)^{\perp}$. Since $p$ is a smooth point of $C$, the ideal $\II_{C,p}$ is generated by a regular sequence of $k$ elements, so both spaces have dimension $k$ and the map above is an isomorphism onto $(T_pC)^{\perp}$. (Equivalently, $\II_C\otimes\kappa(p)$ is the fiber at $p$ of the conormal sheaf $\II_C/\II_C^2$, because $\II_{C,p}^2\subset \mathfrak{m}_p\II_{C,p}$.)

Twisting by $\OO(m)$, the quotient $\II_C(m)/\mathfrak{m}_p\II_C(m)$ is a skyscraper sheaf supported at $p$, with stalk $\II_C(m)\otimes\kappa(p)$, and we have the exact sequence
$$0\longrightarrow \mathfrak{m}_p\II_C(m)\longrightarrow \II_C(m)\longrightarrow \II_C(m)\otimes\kappa(p)\longrightarrow 0.$$
By Serre's vanishing theorem \cite[III.5.2]{Har}, $H^1(\Pp^{k+1}, \mathfrak{m}_p\II_C(m))=0$ for $m\gg 0$, whence the induced map on global sections
$$H^0(\Pp^{k+1},\II_C(m))\longrightarrow \II_C(m)\otimes \kappa(p)$$
is surjective. Choose now $g\in H^0(\OO(m))$ with $g(p)\neq 0$. The space on the right hand side of the last map, namely the fiber $\II_C(m)\otimes\kappa(p)$, is canonically $\big(\II_C\otimes\kappa(p)\big)\otimes \OO(m)|_p$, and $g(p)$ is a basis of the line $\OO(m)|_p$; thus $g$ trivializes $\OO(m)$ near $p$ and identifies
$$\II_C(m)\otimes\kappa(p)\;\simeq\;\II_C\otimes\kappa(p)\;\simeq\;(T_pC)^{\perp}.$$
Concretely, this identification sends a section $s$ to $d(s/g)(p)=ds(p)/g(p)$, the quotient $s/g$ being a rational function which is regular near $p$ and vanishes along $C$ there. By surjectivity we may therefore choose sections $s_2,\dots,s_{\ell}\in H^0(\II_C(m))$ whose differentials $ds_i(p)$ are linearly independent in $(T_pC)^{\perp}$. For $i=2,\dots,\ell$, let $\FF_i$ be the foliation defined by the rational first integral $s_i/g$: it is regular at $p$, since $d(s_i/g)(p)=ds_i(p)/g(p)\neq 0$, its local leaf through $p$ is the germ of the hypersurface $H_i=\{s_i=0\}$, which contains $C$, so (c) holds, and the conormals $ds_i(p)$ are linearly independent in $(T_pC)^{\perp}$.

For $\FF_1$, in case (i) choose linear forms $H, H'$ with $H(p)=0$, $H'(p)\neq 0$ and $T_p\{H=0\}\not\supset T_pC$, and let $\FF_1$ be the foliation defined by $R_0=H/H'$; it is regular at $p$ and transverse to $C$ at $p$, giving (b), and its conormal $dR_0(p)$ does not lie in $(T_pC)^{\perp}$, so that, together with the previous paragraph, (a) holds. Given any non-constant rational first integral $R$ of $\FF_1$: $R$ is constant on the generic member of the pencil, which is an irreducible hyperplane; hence $R$ induces a well defined rational map $\phi$ on the parameter line of the pencil, and $R=\phi\circ R_0$.

In case (ii), let $\mathcal{J}$ be a foliation on $\Pp^2$ of degree $\geq 2$ without invariant algebraic curves, e.g.\ Jouanolou's foliation \cite{Jou}; in particular $\mathcal{J}$ has no non-constant rational first integral, since the closure of a generic fiber of a first integral would be an invariant curve. Let $\Lambda:\Pp^{k+1}\dashrightarrow \Pp^2$ be a generic linear projection, and $\FF_1:=\Lambda^*\mathcal{J}$. For generic $\Lambda$, the point $p$ is not on the center of projection, $\Lambda$ is a submersion at $p$, $\mathcal{J}$ is regular at $\Lambda(p)$, and the conormal $(d\Lambda_p)^*\big(\eta(\Lambda(p))\big)$ avoids any finite set of ``bad'' directions; hence $\FF_1$ is regular at $p$ and conditions (a), (b) hold for generic choices. Finally, $\FF_1$ has no non-constant rational first integral: the fibers of $\Lambda$ are tangent to $\FF_1$ (they are contained in the kernel of the defining form), and the generic fiber is irreducible, so any rational first integral $R$ of $\FF_1$ is constant on the generic fiber of $\Lambda$ and therefore factors as $R=\overline{R}\circ \Lambda$ with $\overline R$ rational on $\Pp^2$; then $\overline R$ is a first integral of $\mathcal{J}$, hence constant.
\end{proof}

The last ingredient is the behavior of the normal bundle of a curve under blow-up at a smooth point.

\begin{lemma}\label{lem:normal}
Let $X$ be a complex manifold of dimension $n\geq 2$, $\widehat C\subset X$ a curve, smooth at a point $\hat p\in \widehat C$, and let $\rho: \widetilde X:=Bl_{\hat p}X\to X$ be the blow-up at $\hat p$, with exceptional divisor $E$ and strict transform $\widetilde C$ of $\widehat C$. Then $\rho|_{\widetilde C}:\widetilde C\to \widehat C$ is a biholomorphism and, if $\widehat C$ is smooth,
$$N_{\widetilde C|\widetilde X}\;\simeq\; (\rho|_{\widetilde C})^*N_{\widehat C|X}\otimes \OO_{\widetilde C}(-\tilde p), \qquad \tilde p:=\widetilde C\cap E.$$
In particular $\deg N_{\widetilde C|\widetilde X}=\deg N_{\widehat C|X}-(n-1)$.
\end{lemma}
\begin{proof}
The first statement is standard, since $\widehat C$ is smooth at $\hat p$. The differential of $\rho$ induces a morphism of vector bundles $d\rho: N_{\widetilde C|\widetilde X}\to (\rho|_{\widetilde C})^*N_{\widehat C|X}$, which is an isomorphism outside $\tilde p$. In suitable local coordinates $(z_1,\dots,z_n)$ at $\hat p$ we have $\widehat C=\{z_2=\dots=z_n=0\}$, and in the chart $(t,u)$ of the blow-up, $\widetilde C=\{u=0\}$ with $\rho(t,u)=(t,tu_2,\dots,tu_n)$. The classes of $\partial_{u_2},\dots,\partial_{u_n}$ trivialize $N_{\widetilde C}$ near $\tilde p$, the classes of $\partial_{z_2},\dots,\partial_{z_n}$ trivialize $N_{\widehat C}$ near $\hat p$, and in these trivializations $d\rho$ is the matrix $t\cdot \mathrm{Id}$. Hence $d\rho$ vanishes exactly to order one at $\tilde p$, in every normal direction, and the induced map $N_{\widetilde C}\otimes\OO(\tilde p)\to (\rho|_{\widetilde C})^*N_{\widehat C}$ is an isomorphism of vector bundles. The degree formula follows by taking determinants, since $\rank N=n-1$.
\end{proof}

\begin{proof}[Proof of Theorem \ref{thm:A}]
Write $n=k+1$ and let $C\subset \Pp^{k+1}$, $\ell\in\{0,\dots,k+1\}$ be given; fix a smooth point $p\in C$.

\medskip\noindent
{\it Step 1: the chart at $p$.}
If $\ell\geq 1$, let $\FF_1,\dots,\FF_{\ell}$ be the foliations given by Lemma \ref{lem:adapted}, and apply Proposition \ref{ell=>1}, with $\Delta_2$ a coordinate neighborhood of $p$, to obtain a biholomorphism $\Psi:\mathbb{D}^{k+1}\to \Delta\ni p$ with the properties (i), (ii), (iii) stated there. If $\ell=0$, apply instead Proposition \ref{ell=0}, obtaining $\Psi$ with property (i) and with none of the coordinate foliations extending. In both cases, set
$$\tilde V:=\Delta, \qquad \tilde\psi:=\Psi^{-1}:\tilde V\to V:=\mathbb{D}^{k+1}.$$
Then $\tilde\psi(p)=0$ and $\tilde\psi(C\cap \tilde V)=\{y=0\}$, and $V$ is a polydisc containing $\overline{U}$, as required in Section \ref{sec:gluing} (recall $\overline U\subset \overline{\mathbb{D}}^{k+1}_r$ with $r<1$).

\medskip\noindent
{\it Step 2: the gluing.}
Perform the construction of Section \ref{sec:gluing} with these choices of $\tilde\psi,\tilde\varphi, T_C, T_E$, obtaining the Hausdorff complex manifold $S_0=T_C\sqcup_{\tilde h}T_E$, which contains the curve $C$ and the divisor $E$ meeting transversely at the point $p=q$.

\medskip\noindent
{\it Step 3: $\Distr(S_0)=\Fol(S_0)$ has exactly $\ell$ elements.}
First, each $\FF_i$ ($1\leq i\leq \ell$) induces a foliation $\widehat\FF_i$ on $S_0$. Indeed, consider $\FF_i|_{T_C}$ on the piece $T_C$; on the overlap $\tilde h(\tilde W)$, in the chart $\tilde\psi$, the foliation $\FF_i$ is the coordinate foliation $\{dz_i=0\}$ by Proposition \ref{ell=>1}(ii); since $h$ preserves the coordinate foliations, the pullback $\tilde h^*(\FF_i|_{\tilde h(\tilde W)})$ is the coordinate foliation $\{dt=0\}$ (if $i=1$) or $\{du_{i}=0\}$ (if $i\geq2$) of the chart $(t,u)$ restricted to $\tilde W$. By Lemma \ref{lem:Eside}(2), the latter is the restriction to $\tilde W$ of a foliation defined on all of $Bl_0\,\CC^{k+1}$, in particular on $T_E$. The two pieces agree on the overlap, and hence define a foliation $\widehat\FF_i\in\Fol(S_0)$. The foliations $\widehat\FF_1,\dots,\widehat\FF_{\ell}$ are pairwise distinct, since their conormal directions at $p$ are linearly independent by the general position hypothesis (a) of Proposition \ref{ell=>1}.

Conversely, let $\DD\in \Distr(S_0)$. By Theorem \ref{thm:coordinate}, $\DD|_{T_C}$ extends to a distribution $\widehat\DD$ on $\Pp^{k+1}$ and, in the chart $\tilde\psi=\Psi^{-1}$, the distribution $\tilde\psi_*(\widehat\DD|_{\tilde V})=\Psi^*\widehat\DD$ is one of the coordinate distributions $\{dz_j=0\}$, $j=1,\dots,k+1$; moreover it is integrable, so $\widehat\DD$ is a foliation extending $\Psi_*\{dz_j=0\}$ to $\Pp^{k+1}\supset\Delta_2$. If $\ell=0$, this contradicts Proposition \ref{ell=0}(ii); hence $\Distr(S_0)=\varnothing$. If $\ell\geq 1$ and $j>\ell$, this contradicts Proposition \ref{ell=>1}(iii), since $\Delta_2$ (and a fortiori $\Pp^{k+1}$) contains a neighborhood of $\overline\Delta$. Hence $j\leq \ell$, and $\widehat\DD$ and $\FF_j$ are two global objects on $\Pp^{k+1}$ that agree on the open set $\Delta$; by $(\star)$ they agree everywhere, so $\DD|_{T_C}=\FF_j|_{T_C}$. Finally, $\DD$ and $\widehat\FF_j$ are two distributions on the connected manifold $S_0$ that agree on the open subset $T_C$; again by $(\star)$, $\DD=\widehat\FF_j$. Therefore
$$\Distr(S_0)=\Fol(S_0)=\{\widehat\FF_1,\dots,\widehat\FF_{\ell}\}.$$

\medskip\noindent
{\it Step 4: contraction of $E$.}
Inside $S_0$, the divisor $E$ has a neighborhood biholomorphic to a neighborhood of the exceptional divisor in $Bl_0\,\CC^{k+1}$, hence normal bundle $\OO_{\Pp^k}(-1)$; it can therefore be contracted to a smooth point (see \cite{FN}; for $k=1$ this is Castelnuovo--Grauert \cite{Gra}). Concretely, define
$$S:=\big(S_0\setminus E\big)\ \sqcup_{\pi}\ \mathbb{D}^{k+1}_{\varepsilon},$$
gluing the open subset $\pi^{-1}(\mathbb{D}^{k+1}_{\varepsilon})\setminus E$ of $S_0\setminus E$ with $\mathbb{D}^{k+1}_{\varepsilon}\setminus\{0\}$ via $\pi$. An argument analogous to the Hausdorff verification of Section \ref{sec:gluing} shows that $S$ is a Hausdorff complex manifold; there is a natural holomorphic map $\rho: S_0\to S$ contracting $E$ to the point $0$ and biholomorphic elsewhere, which realizes $S_0$ as the blow-up of $S$ at $0$. The curve $C\subset S_0$ maps to a curve $\widehat C:=\rho(C)\subset S$; since $C$ meets $E$ transversely at the single (smooth) point $p$, the map $\rho|_C: C\to \widehat C$ is a biholomorphism and $\widehat C$ is a compact curve of $S$ passing through $0$.

Let us check that $\widehat C$ is an \emph{embedded copy} of $C$ in the sense fixed in the introduction, which is where the possible singularities of $C$ are taken into account. Let $x\in C$. If $x\neq p$, then $x$ has a neighborhood $\Omega\subset T_C$ in $\Pp^{k+1}$ with $p\notin \Omega$ and $\Omega\cap E=\varnothing$; the composition $\Omega\hookrightarrow S_0\xrightarrow{\ \rho\ } S$ is then a biholomorphism onto an open subset of $S$ carrying $C\cap\Omega$ onto $\widehat C\cap\rho(\Omega)$, as required -- note that $\rho$ is biholomorphic away from $E$ and that $\Omega$ is an open subset of $\Pp^{k+1}$, so the singularities of $C$ other than $p$ are reproduced verbatim in $S$. If $x=p$, recall that $p$ is a smooth point of $C$; near $\rho(p)=0$ the manifold $S$ is the polydisc $\mathbb{D}^{k+1}_{\varepsilon}$ and $\widehat C$ is a smooth curve germ through the origin, so the required ambient biholomorphism exists because any two smooth curve germs in $\CC^{k+1}$ are equivalent.

Pull-back by $\rho$ gives an injection $\Distr(S)\to \Distr(S_0)$ (and similarly for foliations), and this map is a bijection: given $\DD_0\in\Distr(S_0)$, its direct image on $S\setminus\{0\}$ extends through the point $0$ (the coefficients of a defining $1$-form extend by Hartogs, as in Lemma \ref{lem:Eside}(1)), producing $\DD\in \Distr(S)$ with $\rho^*\DD=\DD_0$. Therefore
$$\#\Distr(S)=\#\Fol(S)=\ell,$$
which proves part (1) of the theorem.

\medskip\noindent
{\it Step 5: the germ statement.}
Let $S'\subset S$ be a connected open neighborhood of $\widehat C$. The restriction $\Fol(S)\to\Fol(S')$ is injective by $(\star)$. For surjectivity, let $\GG'\in \Fol(S')$ (or $\Distr(S')$) and consider $\rho^*\GG'$ on $\widetilde S':=\rho^{-1}(S')$, a connected neighborhood of $C\cup E$ in $S_0$. The restriction of $\rho^*\GG'$ to a connected neighborhood of $C$ contained in $\widetilde S'\cap T_C$ extends to $\Pp^{k+1}$ by \cite[Theorem 3.1]{Ros}, and the proof of Theorem \ref{thm:coordinate} applies verbatim -- the contradiction there only uses the values of the distribution near $\tilde\varphi^{-1}(\{0\}\times\overline{\mathbb{D}}^k)\subset E$, which is contained in $\widetilde S'$. As in Step 3, we conclude that $\rho^*\GG'$ agrees with some $\widehat\FF_j$ ($j\leq \ell$) on an open set, hence on all of $\widetilde S'$, and $\GG'$ is the restriction of the corresponding foliation of $S$. This proves part (2).

\medskip\noindent
{\it Step 6: the normal bundle.}
Assume $C$ smooth. Since $T_C$ is an open neighborhood of $C$ in $\Pp^{k+1}$, we have $N_{C|S_0}\simeq N_{C|\Pp^{k+1}}$. By Step 4, $S_0$ is the blow-up of $S$ at $0\in \widehat C$ and $C$ is the strict transform of $\widehat C$; Lemma \ref{lem:normal} then gives
$$N_{C|\Pp^{k+1}}\simeq N_{C|S_0}\simeq (\rho|_C)^*N_{\widehat C|S}\otimes \OO_C(-p),$$
that is, $N_{\widehat C|S}\simeq N_{C|\Pp^{k+1}}\otimes \OO_C(p)$ under the identification $\rho|_C$. Taking degrees of determinants, and recalling that $N_{C|\Pp^{k+1}}$ has rank $k$, we get $\deg N_{\widehat C|S}=\deg N_{C|\Pp^{k+1}}+k$, which is part (3) with $n=k+1$.

Finally, let us make this number explicit in terms of the classical invariants of $C$. Writing $d$ for the degree and $g$ for the genus of the smooth curve $C\subset\Pp^n$, the normal bundle sequence
$$0\longrightarrow T_C\longrightarrow T_{\Pp^{n}}|_C\longrightarrow N_{C|\Pp^{n}}\longrightarrow 0$$
gives $\deg N_{C|\Pp^n}=\deg\big(T_{\Pp^n}|_C\big)-\deg T_C$. Since $c_1(T_{\Pp^n})=(n+1)H$ with $H$ the hyperplane class, we have $\deg(T_{\Pp^n}|_C)=(n+1)(H\cdot C)=(n+1)d$, while $\deg T_C=2-2g$. Hence
$$\deg N_{C|\Pp^n}=(n+1)d+2g-2, \qquad \deg N_{\widehat C|S}=(n+1)d+2g+n-3.$$
For $n=2$ a smooth plane curve has $g=\frac{(d-1)(d-2)}{2}$, so $\deg N_{C|\Pp^2}=C\cdot C=d^2$ and $\widehat C\cdot\widehat C=d^2+1$, in accordance with \cite{FLS}. For a line in $\Pp^n$ ($d=1$, $g=0$) one recovers $\deg N_{C|\Pp^n}=n-1$, as it must be since $N_{C|\Pp^n}\simeq\OO_{\Pp^1}(1)^{\oplus(n-1)}$.
\end{proof}

\begin{remark}
For $n=2$ the theorem produces, for any projective curve $C\subset\Pp^2$ and any $\ell\in\{0,1,2\}$, a germ of surface neighborhood of (a copy of) $C$ with exactly $\ell$ foliations, with normal bundle $N_{C|\Pp^2}\otimes\OO_C(p)$; compare with \cite{FL}, where germs $(S,C)$ with $C$ rational of arbitrary positive self-intersection and without (even formal) foliations are constructed, and with \cite{FLS}. The gluing performed here is a higher dimensional, ``wilder'' version of the constructions in \cite{FLS,Lv1}.
\end{remark}

\section{The field of meromorphic functions}\label{sec:functions}

We start with the elementary observation relating meromorphic functions and foliations, already alluded to in the introduction.

\begin{prop}\label{prop:pencil}
Let $S$ be a connected complex manifold. If there exist $f,g\in \MM(S)$ with $df\wedge dg\not\equiv 0$, then $\Fol(S)$ is infinite. Consequently, if $\Fol(S)$ is finite, then any two meromorphic functions on $S$ are functionally dependent.
\end{prop}
\begin{proof}
For $t\in \CC$, let $\FF_t\in\Fol(S)$ be the foliation defined by $d(f+tg)=df+t\,dg$. If $\FF_t=\FF_{t'}$, the forms $df+t\,dg$ and $df+t'\,dg$ are proportional at every point, so
$$0\equiv (df+t\,dg)\wedge(df+t'\,dg)=(t'-t)\,df\wedge dg,$$
forcing $t=t'$. Hence $t\mapsto \FF_t$ is injective and $\Fol(S)$ is infinite.
\end{proof}

We now compute the field of meromorphic functions of the manifolds of Theorem \ref{thm:A}. We use freely the notation of Sections \ref{sec:gluing} and \ref{sec:proofA}; in particular $S$ denotes the manifold constructed there, $S_0=T_C\sqcup_{\tilde h}T_E$ its blow-up at $0$, and, when $\ell\geq 1$, $\FF_1$ denotes the foliation of Lemma \ref{lem:adapted}, transverse to the curve at $p$. Note that $\rho^*:\MM(S)\to \MM(S_0)$ is an isomorphism, by Levi extension through the point $0$ (Lemma \ref{lem:Eside}(1) applied downstairs); we may therefore work on $S_0$.

\begin{lemma}\label{lem:functions}
Restriction to $T_C$, followed by the extension theorem of Rossi, defines an injective field homomorphism $\MM(S)=\MM(S_0)\hookrightarrow \CC(\Pp^{k+1})$, $F\mapsto R$. Moreover, a rational function $R\in\CC(\Pp^{k+1})$ lies in the image if and only if $R\circ \Psi$ depends only on the variable $z_1$. Consequently:
\begin{enumerate}[(1)]
\item if $\ell=0$, then $\MM(S)=\CC$;
\item if $\ell\geq 1$, then $\MM(S)=\CC\cup\{\text{rational first integrals of } \FF_1\}$.
\end{enumerate}
\end{lemma}
\begin{proof}
Given $F\in \MM(S_0)$, its restriction to $T_C$ is a meromorphic function on a connected neighborhood of the compact curve $C\subset \Pp^{k+1}$, and therefore extends to a rational function $R$ on $\Pp^{k+1}$ \cite[Theorem 3.1]{Ros}. The map $F\mapsto R$ is injective, by the identity theorem: two meromorphic functions on the connected manifold $S_0$ which agree on the non empty open set $T_C$ agree everywhere. It is clearly a field homomorphism.

Let now $F\in\MM(S_0)$ with associated $R$, and set $\varrho:=R\circ \Psi$, a meromorphic function on $\mathbb{D}^{k+1}$ which, since $R$ is rational and $\overline\Delta\subset \Delta_2$, is the restriction of a meromorphic function on a neighborhood of $\overline{\mathbb{D}}_r^{k+1}$ in the following sense: $R$ is meromorphic on $\Delta_2\supset\overline\Delta$ and, on $\overline U\subset \overline{\mathbb{D}}_r^{k+1}\subset\mathbb{D}^{k+1}$, $\varrho$ is meromorphic on a neighborhood of $\overline U$. On the other side, $F|_{T_E}$ comes from the ball: by Lemma \ref{lem:Eside}(1) there are $\varepsilon\in(0,1)$ and $G\in \MM(\mathbb{D}^{k+1}_{\varepsilon})$ with $F=G\circ \pi$ on $\pi^{-1}(\mathbb{D}_{\varepsilon}^{k+1})$. Comparing the two expressions of $F$ on the overlap, via $\tilde h$, we obtain the identity
\begin{equation}\label{eq:matching}
G(x, xy_1,\dots,xy_{k})\;=\;\varrho\big(f(x), g(y_1),\dots,g(y_{k})\big)
\end{equation}
for all $(x,y)\in \mathbb{D}\times\mathbb{D}^k$ with $|x|<\varepsilon$. Here we have written the chart $\tilde\varphi$ on $\tilde W$ in its own coordinates $(x,y_1,\dots,y_k)$, so that $\pi$ reads $\pi(x,y)=(x,xy_1,\dots,xy_k)$; recall that $\varrho$ is a function of $(z_1,\dots,z_{k+1})=(x,y_1,\dots,y_k)$.

\noindent{\bf Claim:} $\varrho$ depends only on $z_1=x$.
Fix $j\in\{1,\dots,k\}$, a value $a\in\mathbb{D}^*_{\varepsilon}$ and constants $c_i\in \mathbb{D}$ for $i\neq j$, all of them generic. Freezing every variable except the $j$-th vertical one, consider the two one variable functions
\begin{align*}
w(\xi)&:=G\big(a,\,ac_1,\dots,ac_{j-1},\,a\xi,\,ac_{j+1},\dots,ac_{k}\big),\\
\varrho_j(\zeta)&:=\varrho\big(f(a),\,g(c_1),\dots,g(c_{j-1}),\,\zeta,\,g(c_{j+1}),\dots,g(c_k)\big).
\end{align*}
The function $\varrho_j$ is meromorphic on a neighborhood of $\overline D$, because the point $\big(f(a),g(c_1),\dots,g(c_k)\big)$ lies in $U$ and $\varrho$ is meromorphic on a neighborhood of $\overline{U}$; for a generic choice of the parameters, $\varrho_j\not\equiv\infty$. By \eqref{eq:matching} we have $w(\xi)=\varrho_j\big(g(\xi)\big)$ for $\xi\in \mathbb{D}$. On the other hand $G$ is meromorphic on $\mathbb{D}^{k+1}_{\varepsilon}$ and $|a|<\varepsilon$, so $w$ is meromorphic on the disc $\{|\xi|<\varepsilon/|a|\}$, hence on a neighborhood of $\overline{\mathbb{D}}$. Then $g^*(d\varrho_j)=dw$ extends meromorphically to a neighborhood of $\overline{\mathbb{D}}$, and Lemma \ref{Lema-principal}(2), applied to the meromorphic $1$-form $d\varrho_j$ on a neighborhood of $\overline D$, forces $d\varrho_j\equiv0$: the function $\varrho_j$ is constant. Letting the parameters $a$ and $c_i$ vary, we conclude that $\partial\varrho/\partial y_j$ vanishes on the open set $U$, hence identically, and the Claim is proved.

Thus $\varrho=\varrho(z_1)$ for every $F\in \MM(S_0)$. Conversely, suppose $R\in\CC(\Pp^{k+1})$ satisfies $R\circ\Psi=\varrho(z_1)$. Define $\gamma:=\varrho\circ f\in\MM(\mathbb{D})$ and
$$F_E:=\gamma\circ z_1\circ \pi \ \in \MM(T_E),$$
which makes sense because $\pi(T_E)\subset \mathbb{D}^{k+1}$ by \eqref{eq:TEconditions}. On the overlap $\tilde W$, in the coordinates $(x,y)$ of the chart $\tilde\varphi$, we have $F_E=\gamma(x)=\varrho(f(x))$, which is precisely $R\circ\tilde h$ read in that chart: indeed $R\circ\tilde h=\varrho\circ h$ and $\varrho\big(h(x,y)\big)=\varrho\big(f(x),g(y_1),\dots,g(y_k)\big)=\varrho(f(x))$, since $\varrho$ does not depend on the variables $y_1,\dots,y_k$. Hence $R|_{T_C}$ and $F_E$ glue to a global $F\in \MM(S_0)$. Clearly $F$ is mapped to $R$ by the injection $\MM(S_0)\hookrightarrow \CC(\Pp^{k+1})$.

Finally we prove the assertions (1) and (2). If $\ell=0$, the foliation $\{dR=0\}$ would then be a foliation on $\Delta_2$ extending $\Psi_*\{dz_1=0\}$, contradicting Proposition \ref{ell=0}(ii); hence only constants occur and (1) follows. If $\ell\geq1$, the level sets $\Psi(\{z_1=c\})$ are the leaves of $\FF_1|_{\Delta}$ by Proposition \ref{ell=>1}(ii); thus $dR\wedge \omega_1\equiv 0$ on $\Delta$, where $\omega_1$ defines $\FF_1$, and since $\Pp^{k+1}$ is connected the identity theorem gives $dR\wedge\omega_1\equiv 0$ everywhere, i.e.\ $R$ is a rational first integral of $\FF_1$. Conversely, if $R$ is a rational first integral of $\FF_1$, then $R\circ \Psi$ is a meromorphic function on $\mathbb{D}^{k+1}$ constant on the fibers of $z_1$, hence of the form $\varrho(z_1)$. This proves (2).
\end{proof}

\begin{proof}[Proof of Theorem \ref{thm:B}]
Part (1) is Lemma \ref{lem:functions}(1). For part (2), suppose $\ell\geq 1$. If, in Lemma \ref{lem:adapted}, we choose $\FF_1$ as in case (ii) -- without rational first integrals -- then Lemma \ref{lem:functions}(2) gives $\MM(S)=\CC$. If instead we choose $\FF_1$ as in case (i) -- the pencil of hyperplanes with first integral $R_0$ -- then, by the same lemma and by the description of the first integrals of a linear pencil,
$$\MM(S)=\CC\cup\{\phi\circ R_0:\ \phi\in\CC(w) \text{ non-constant}\}=\CC(R_0|_S),$$
a purely transcendental extension of transcendence degree one, generated by the meromorphic function $F\in\MM(S)$ corresponding to $R_0$, which is a first integral of $\FF_1$ on $S$. Moreover, for $i\geq2$, a non-constant meromorphic first integral of $\FF_i$ on $S$ would be, by Lemma \ref{lem:functions}, a rational first integral of $\FF_1$; but then $\FF_i=\FF_1$, since both foliations would coincide with the foliation defined by its differential, contradicting general position. (Note that on $\Pp^{k+1}$ each $\FF_i$, $i\geq 2$, \emph{does} admit the rational first integral $s_i/g$; these functions simply do not survive on $S$.)

For part (3): by Lemma \ref{lem:functions}, any two elements of $\MM(S)$ are rational functions of $\Pp^{k+1}$ which are constant or first integrals of the same one-codimensional foliation, hence functionally dependent ($dR\wedge dR'\equiv 0$); for rational functions on $\Pp^{k+1}$, functional dependence implies algebraic dependence (the map $(R,R'):\Pp^{k+1}\dashrightarrow \CC^2$ has rank at most one, so its image is contained in an algebraic curve). Hence $\trdeg_{\CC}\MM(S)\leq 1$. Alternatively, this follows from Proposition \ref{prop:pencil} together with $\#\Fol(S)=\ell<\infty$. Finally, the statements hold for every connected neighborhood $S'$ of $\widehat C$: the proofs of Lemma \ref{lem:functions} use only the restriction of $F$ to a neighborhood of $C$ (Rossi) and to a neighborhood of $E$ (Lemma \ref{lem:Eside}), both available inside $\rho^{-1}(S')$, and restriction $\MM(S)\to\MM(S')$ is injective and, by the same computation, surjective.
\end{proof}

\begin{remark}
In \cite{FLS}, surfaces containing curves with positive self-intersection and with field of meromorphic functions of transcendence degree $0$, $1$ \emph{or} $2$ are constructed. By Proposition \ref{prop:pencil}, the examples there with transcendence degree $2$ necessarily carry infinitely many foliations -- indeed they carry all the foliations induced by pencils of rational functions of $\Pp^2$. Theorem \ref{thm:B} shows that, conversely, imposing finitely many foliations caps the transcendence degree at $1$. For a study of the possible fields of meromorphic functions on neighborhoods of rational curves in surfaces, see also \cite{Lv2}.
\end{remark}

\begin{remark}
The manifolds of Theorem \ref{thm:A} with $\ell=0$ have no non-constant meromorphic functions (Theorem \ref{thm:B}(1)); for $C$ rational and $n=2$, germs with this property (and with arbitrary normal degree $d>0$) were first constructed in \cite{FL}, and for higher dimensional submanifolds with ample normal bundle in \cite{FLP}. The novelty here is the simultaneous control of the number of foliations, of distributions, and of the function field, in any dimension and for any projective curve.
\end{remark}

\section{A surface with exactly two foliations}\label{sec:pereira}

We present now a different construction, explained to us by Jorge Vit\'orio Pereira, of a surface containing a rational curve and carrying exactly two foliations. Let $\mathcal{F}_1$ and $\mathcal{F}_2$ be codimension one foliations on $\mathbb{P}^3$, defined by homogeneous $1$-forms $\omega_1$ and $\omega_2$, and denote by $\mathcal{G}=\mathcal{F}_1 \cap \mathcal{F}_2$ the codimension two foliation defined by $\omega_1\wedge\omega_2=0$. We assume that:
\begin{enumerate}
\item $\FF_1$ is not singularly transversely affine;
\item $\FF_2$ is not the pull-back of a foliation on a projective surface by a dominant rational map;
\item the form $\omega_1\wedge \omega_2$ is saturated (its zero set has codimension $\geq 2$).
\end{enumerate}
Recall that a foliation $\FF$ is \emph{singularly transversely affine} when it can be defined by a rational $1$-form $\omega$ for which there is a closed rational $1$-form $\vartheta$ with $d\omega=\vartheta\wedge\omega$; the property does not depend on the choice of $\omega$, since $d(g\omega)=\big(\tfrac{dg}{g}+\vartheta\big)\wedge(g\omega)$.

\begin{prop}\label{prop:twofol}
There is no codimension one foliation other than $\mathcal{F}_1$ and $\mathcal{F}_2$ containing $\mathcal{G}$.
\end{prop}

\begin{proof}
Let $\mathcal{H}=\{\omega=0\}$ be a foliation containing $\mathcal{G}$, with $\HH\neq\FF_1,\FF_2$. Then $\omega\wedge \omega_i$ also defines $\mathcal{G}$, so $\omega\wedge \omega_i=P_i\, \omega_1 \wedge \omega_2$ for some homogeneous polynomials $P_i$, $i=1,2$, by (3). Then $(\omega + P_1\omega_2)\wedge \omega_1 = 0$, which implies that we can write $\omega = \dfrac{A}{B}\omega_1 - P_1 \omega_2$ for some homogeneous polynomials $A$ and $B$ (two $1$-forms with vanishing wedge are proportional over the field of rational functions). Taking the wedge product with $\omega_2$ and comparing both sides we get
$$
P_2\, \omega_1 \wedge \omega_2= \frac AB\, \omega_1\wedge \omega_2 \ \Longrightarrow\ \frac AB = P_2 \ \Longrightarrow\ \omega = P_2 \omega_1 - P_1 \omega_2.
$$
Since $\HH\neq\FF_1,\FF_2$, we have $P_1P_2\not\equiv0$. The forms $P_2\omega_1$, $P_1\omega_2$ and $\omega=P_2\omega_1-P_1\omega_2$ are integrable, and the integrability of the member $\eta_t=P_2\omega_1+tP_1\omega_2$ of the pencil they span is a quadratic condition in $t$; vanishing at $t=0$, $t=\infty$ and $t=-1$, it vanishes identically. Thus $\{\eta_t\}_{t\in\Pp^1}$ is a pencil of integrable $1$-forms, and it is non degenerate because $P_2\omega_1\wedge P_1\omega_2=P_1P_2\,\omega_1\wedge\omega_2\not\equiv0$.

We may therefore apply Cerveau's theorem \cite{Cer}, in the form given in \cite[Theorem A]{BP}: either every member of the pencil is singularly transversely affine, or there are a projective surface $Y$, a dominant rational map $\varpi:\Pp^3\dashrightarrow Y$ and foliations $\GG_t$ on $Y$ with $\{\eta_t=0\}=\varpi^*\GG_t$ for all $t$. The first alternative gives in particular that $\FF_1$, which is defined by the member $P_2\omega_1$ corresponding to $t=0$, is singularly transversely affine, contradicting (1). The second gives $\FF_2=\varpi^*\GG_{\infty}$, contradicting (2).
\end{proof}

Now we follow the argument of \cite{FLP}. Recall that $\GG$ is a one dimensional foliation on $\Pp^3$.

\begin{definition}\label{def:weaktransv}
A line $L\subset \Pp^3$ is \emph{weakly transverse} to $\GG$ if
$$L\cap \mathrm{Sing}(\GG)=\varnothing \qquad\text{and}\qquad T_pL\cap T_p\GG=\{0\}\ \ \text{ for \emph{every} } p\in L.$$
\end{definition}

The terminology is meant to stress that this is \emph{not} transversality in the usual sense: both $L$ and $\GG$ are one dimensional, so $T_pL\oplus T_p\GG$ is a plane in the three dimensional space $T_p\Pp^3$ and the two tangent lines never span it. What the condition does say is that $L$ is nowhere tangent to $\GG$, and it is essential that it holds at every point of $L$, not merely at the generic one.

Such lines exist. On the one hand, $\mathrm{Sing}(\GG)$ has dimension at most one, so the lines meeting it form a family of dimension at most $1+2=3$ inside the four dimensional Grassmannian $\mathbb{G}(1,3)$. On the other hand, through each point $p\in \Pp^3\setminus\mathrm{Sing}(\GG)$ there passes exactly one line tangent to $\GG$ at $p$, namely the one with direction $T_p\GG$; hence the lines tangent to $\GG$ at some point also form a family of dimension at most three. A generic line $L$ therefore avoids both families and is weakly transverse to $\GG$.

Fix such an $L$. Since $\GG$ is regular along the compact curve $L$, we may cover $L$ by finitely many foliated charts for $\GG$, in each of which the local leaf space is a smooth two dimensional polydisc, the changes of charts being biholomorphisms between open subsets of such polydiscs. Choosing a tubular neighborhood $W$ of $L$ small enough that each leaf of $\GG|_W$ is contained in a single plaque of one of these charts, the leaf space
$$\Ss:=W/\GG|_W, \qquad \pi: W\longrightarrow \Ss,$$
is a Hausdorff smooth complex surface and $\pi$ is a holomorphic submersion. This is precisely where weak transversality is used: the kernel of $d\pi_p$ is $T_p\GG$, so the condition $T_pL\cap T_p\GG=\{0\}$ at \emph{every} $p\in L$ says exactly that $\pi|_L$ is an immersion. Shrinking $W$ further if necessary, no plaque meets $L$ twice, so $\pi|_L$ is also injective and $C:=\pi(L)$ is a smooth rational curve embedded in $\Ss$. Any foliation $\HH$ on $\Ss$ pulls back, $\pi$ being a submersion, to a foliation $\pi^*\HH$ on $W$ containing $\GG|_W$; since $W$ is a neighborhood of the curve $L\subset \Pp^3$, this foliation extends to $\Pp^3$ by \cite[Theorem 3.1]{Ros}, giving rise to a global foliation containing $\mathcal{G}$ (containment holds on $W$, hence everywhere by $(\star)$); by Proposition \ref{prop:twofol}, $\pi^*\HH$ is the restriction of $\FF_1$ or $\FF_2$. Conversely $\FF_1$ and $\FF_2$, containing $\GG$, descend to two distinct foliations on $\Ss$. We deduce that $\Ss$ carries exactly two foliations.

Unlike the surfaces produced by Theorem \ref{thm:A}, here we do not know whether every codimension one distribution on $\Ss$ is integrable. On the other hand, this construction has the interesting feature that the two foliations of $\Ss$ are of algebraic origin, in the sense that they are the quotients by $\pi$ of the restrictions to $W$ of two foliations globally defined on $\Pp^3$ by homogeneous $1$-forms. We stress that this does \emph{not} mean that they are algebraically integrable in the classical sense; quite the contrary, a foliation with a non-constant rational first integral $f$ is defined by the closed rational $1$-form $df$ and is therefore singularly transversely affine, so hypothesis (1) forbids $\FF_1$ from having one.

\end{document}